\documentclass[12pt]{amsart}

\usepackage{epsf}
\usepackage{graphicx}
\usepackage{latexsym}
\usepackage{amsfonts}
\usepackage{amscd}
\usepackage{amssymb}
\usepackage{amsmath}
\usepackage{amsthm}
\usepackage{bm}
\usepackage[lmargin=1.25in,rmargin=1.25in,tmargin=1.25in,bmargin=1in,dvips]{geometry}
\usepackage{color}
\usepackage{pinlabel}
\usepackage{paralist}
\usepackage{enumitem}
\usepackage[all]{xy}
\usepackage{lscape}
\usepackage[unicode, linktocpage, bookmarksnumbered=true,backref=true]{hyperref}

\newtheorem{theorem}{Theorem}[section]
\newtheorem{lemma}[theorem]{Lemma}

\newtheorem{corollary}[theorem]{Corollary}
\newtheorem{proposition}[theorem]{Proposition}

\theoremstyle{definition}

\newtheorem{remark}[theorem]{Remark}

\def\N{\mathbb{N}}
\def\R{\mathbb{R}}

\def\Z{\mathbb{Z}}
\def\Q{\mathbb{Q}}
\def\F{\mathbb{F}}

\def\CFK {\widehat{\operatorname{CFK}}}

\def\CFD {\widehat{\operatorname{CFD}}}
\def\CFA {\widehat{\operatorname{CFA}}}
\def\CFAA {\widehat{\operatorname{CFAA}}}
\def\CFDD {\widehat{\operatorname{CFDD}}}

\def\CFDA {\widehat{\operatorname{CFDA}}}
\def\CFDm {\operatorname{CFD}^-}
\def\CFAm {\operatorname{CFA}^-}

\def\CFKm {\operatorname{CFK}^-}
\def\HFKm {\operatorname{HFK}^-}

\def\gCFKm {\operatorname{gCFK}^-}
\def\gCFK {\operatorname{g \widehat{CFK}}}

\def\DDhat {\widehat{\mathcal{DD}}}

\def\srel {\ul{\mathfrak{s}}}

\def\AA {\mathcal{A}}

\def\CC {\mathcal{C}}
\def\FF {\mathcal{F}}
\def\GG {\mathcal{G}}
\def\HH {\mathcal{H}}

\def\MM {\mathcal{M}}
\def\PP {\mathcal{P}}

\def\ZZ {\mathcal{Z}}
\def\SS{\mathcal{S}}

\def\I{\mathbb{I}}
\def\T{\mathbb{T}}

\def\a {\mathbf{a}}
\def\s {\mathbf{s}}
\def\t {\mathbf{t}}

\def\x {\mathbf{x}}
\def\y {\mathbf{y}}

\def\ul {\underline}

\newcommand{\abs}[1] {\left\lvert #1 \right\rvert}

\newcommand{\gen}[1] {\langle #1 \rangle}

\def\minus{\smallsetminus}
\def\co{\colon\thinspace}

\def\conn{\mathbin{\#}}

\DeclareMathOperator{\im}{im}  
 \DeclareMathOperator{\lk}{lk} \DeclareMathOperator{\pt}{pt}
 \DeclareMathOperator{\Spin}{Spin}
 \DeclareMathOperator{\nbd}{nbd}
\DeclareMathOperator{\PD}{PD} \DeclareMathOperator{\gr}{gr}

\DeclareMathOperator{\dr}{dr} \DeclareMathOperator{\Mor}{Mor}
\DeclareMathOperator{\Vect}{Vect} 
\DeclareMathOperator{\ex}{ex} \DeclareMathOperator{\ind}{ind}

\definecolor{darkblue}{rgb}{0,0,0.5}
\definecolor{darkred}{rgb}{0.5,0,0}
\definecolor{darkgreen}{rgb}{0,0.5,0}

\numberwithin{equation}{section}

\renewcommand{\tilde}{\widetilde}
\renewcommand{\bar}{\overline}
\renewcommand{\hat}{\widehat}

\begin{document}

\title{Non-surjective satellite operators and piecewise-linear concordance}

\author{Adam Simon Levine}
\address{Department of Mathematics \\ Princeton University \\ Princeton, NJ 08540}
\email{asl2@math.princeton.edu}

\begin{abstract}
We exhibit a knot $P$ in the solid torus, representing a generator of first homology, such that for any knot $K$ in the $3$-sphere, the satellite knot with pattern $P$ and companion $K$ is not smoothly slice in any homology $4$-ball. As a consequence, we obtain a knot in a homology $3$-sphere that does not bound a piecewise-linear disk in any homology $4$-ball.
\end{abstract}


\maketitle

\section{Introduction}

A knot $K$ in the boundary of a smooth $4$-manifold $X$ is called \emph{smoothly slice} (in $X$) if it bounds a smoothly embedded disk in $X$, and \emph{topologically slice} if it merely bounds a locally flatly embedded disk (i.e., a continuous embedding with a topological normal bundle). The classical study of knot concordance aims to classify which knots in the $3$-sphere $S^3$ are (topologically or smoothly) slice in the $4$-ball $D^4$. Note that without the requirement of either smoothness or local flatness, this question becomes uninteresting, since every knot $K$ in $S^3$ bounds a piecewise-linear (PL) embedded disk in $D^4$ (and hence in any $4$-manifold with boundary $S^3$), obtained by taking the cone on $K$.

In the 1960s, Zeeman \cite{ZeemanDunce} conjectured that the analogous statement can fail to hold for knots in the boundary of an arbitrary compact, contractible $4$-manifold, making PL concordance of knots in $3$-manifolds other than $S^3$ a nontrivial subject. Akbulut \cite{AkbulutZeeman} proved Zeeman's conjecture in 1991, exhibiting a contractible $4$-manifold $X$ and a knot $J \subset \partial X$ which does not bound any embedded PL disk in $X$. However, this knot does bound a smoothly embedded disk in a different contractible $4$-manifold $X'$ with $\partial X' = \partial X$. The purpose of this paper is to prove a stronger statement:

\begin{theorem} \label{thm:main}
There exist a smooth, compact, contractible $4$-manifold $X$ and a knot $J \subset \partial X$ such that $J$ does not bound a PL disk in $X$ or in any other rational homology $4$-ball $X'$ with $\partial X' = \partial X$. Moreover, $J$ can be assumed to be topologically slice in $X$.
\end{theorem}

Our strategy for proving Theorem \ref{thm:main} is to reduce it to a problem concerning smooth concordance of satellite knots in $S^3$. We begin by reviewing some terminology. Two oriented knots $K_0$ and $K_1$ in $S^3$ are called (smoothly) \emph{concordant} if there is a smoothly embedded annulus in $S^3 \times I$ whose boundary is $-K_0 \times \{0\} \cup K_1 \times \{1\}$. The set of concordance classes of knots in $S^3$ forms a group $\CC$ under connected sum, with identity element given by the concordance class of the unknot. It is often useful to consider weaker notions of concordance as well. We call $K_0$ and $K_1$ \emph{exotically concordant} (or \emph{pseudo-concordant}) if they cobound a smoothly embedded annulus in a smooth $4$-manifold that is homeomorphic to $S^3 \times I$ but perhaps has an exotic smooth structure, and (for any ring $R$) \emph{$R$-homology concordant} if they cobound a smoothly embedded annulus in a smooth manifold with the $R$-homology of $S^3 \times I$. Denote the corresponding groups by $\CC_{\ex}$ and $\CC_{R}$, respectively. There are surjective forgetful maps
\[
\CC \twoheadrightarrow \CC_{\ex} \twoheadrightarrow \CC_{\Z} \twoheadrightarrow \CC_{\Q}.
\]
We say that $K \subset S^3$ is (smoothly) \emph{slice}, \emph{exotically slice}, or \emph{$R$-homology slice} if it represents the trivial element of $\CC$, $\CC_{\ex}$, or $\CC_{R}$, respectively; this is equivalent to $K$ bounding an embedded disk in $D^4$, a contractible $4$-manifold (which must be homeomorphic to $D^4$ by work of Freedman \cite{Freedman4Manifolds}), or an $R$-homology $4$-ball. (If the smooth $4$-dimensional Poincar\'e conjecture is true, then any smooth, compact, contractible $4$-manifold with boundary $S^3$ must be diffeomorphic to $D^4$, so $\CC = \CC_{\ex}$.)

Fix an orientation on $S^1$; this determines a generator of $H_1(S^1 \times D^2;\Z) \cong \Z$. Given an oriented knot $P \subset S^1 \times D^2$, the \emph{winding number} of $P$ is the integer $m$ such that $P$ represents $m$ times this generator. For any knot $K \subset S^3$, the Seifert framing of $K \subset S^3$ determines an embedding of $S^1 \times D^2$ in $S^3$ as a regular neighborhood of $K$, up to isotopy. We define $P(K)$, the \emph{satellite of $K$ with pattern $P$}, as the image of $P$ under this embedding. If $K_0$ is concordant to $K_1$, then $P(K_0)$ is concordant to $P(K_1)$, so $P$ induces a function from each of the groups $\CC, \CC_{\ex}, \CC_R$ to itself, known as a \emph{satellite operator}. (In general, satellite operators are not group homomorphisms.)

Problem 1.45 in Kirby's problem list \cite{KirbyList}, attributed to Akbulut, asks whether there exists a winding number $\pm 1$ satellite operator $P$ for which $P(K)$ is never exotically slice. The following theorem answers a stronger form of this question in the affirmative:

\begin{theorem} \label{thm:nonsurjective}
There exists a pattern knot $P \subset S^1 \times D^2$ with winding number $1$ such that for any knot $K \subset S^3$, $P(K)$ is not slice in any rational homology $4$-ball; that is, the images of the maps on $\CC$, $\CC_{\ex}$, $\CC_\Z$, and $\CC_\Q$ induced by $P$ do not contain $0$.
\end{theorem}

\begin{figure}
\labellist
 \pinlabel (a) at 10 110
 \pinlabel (b) at 165 110
 \pinlabel (c) at 338 110
 \small
 \pinlabel $P$ at 125 56
 \pinlabel $0$ at 125 93
 \pinlabel $J_P$ [b] at 20 68
 \pinlabel $P$ at 300 56
 \pinlabel $0$ at 300 93
 \pinlabel $0$ [l] at 242 55
 \pinlabel $0$ [b] at 191 71
 \pinlabel $-K$ at 177 55
 \pinlabel $P$ at 435 56
 \pinlabel $0$ at 435 93
 \pinlabel $-K$ at 357 55
 \pinlabel $\gamma$ [b] at 397 110
 \large
 \pinlabel $\sim$ at 326 55
\endlabellist
\includegraphics[scale=0.9]{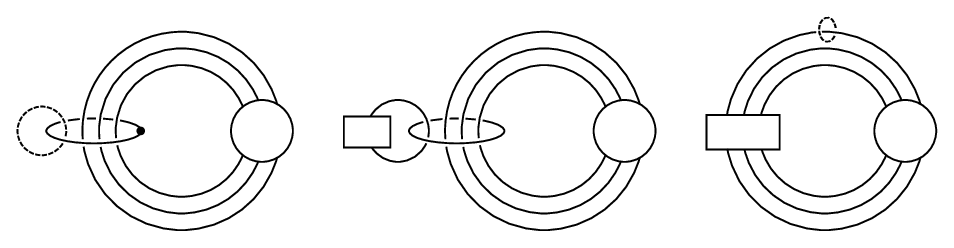}
\caption{(a) Kirby diagram for the Mazur manifold $X_P$, along with the knot $J_P \subset Y_P$. The circle marked $P$ denotes a tangle with an odd number of strands. (b-c) Dehn surgery diagrams showing that $\partial W = S^3_0(P(-K))$, whose first homology is generated by the curve $\gamma$. In (c), the strands passing through the box marked $-K$ are $0$-framed parallels of $-K$.}
\label{fig:kirby}
\end{figure}

Note that to prove Theorem \ref{thm:nonsurjective}, it suffices to find a winding-number-one satellite operator $Q$ that is non-surjective on $\CC_{\Q}$, i.e., that there exists a knot $L \subset S^3$ such that $L$ is not rational homology concordant to $Q(K)$ for any $K \subset S^3$. Then $P = Q \conn {-L} \subset S^1 \times D^2$ satisfies the conclusion of Theorem \ref{thm:nonsurjective}.

Before introducing the example that proves Theorem \ref{thm:nonsurjective}, we show how this result implies the first part of Theorem \ref{thm:main}. For any pattern knot $P \subset S^1 \times D^2$, let $\tilde P \subset S^3$ be the knot obtained by applying $P$ to the unknot. Let $\lambda_P$ be the framing of $P$ that corresponds to the Seifert framing of $\tilde P$. Viewing $P$ as a knot in the boundary of $S^1 \times D^3$, let $X_P$ be the manifold obtained by attaching a $2$-handle to $S^1 \times D^3$ along $P$ with framing $\lambda_P$, and let $Y_P = \partial X_P$. (A schematic Kirby diagram for $X_P$ is shown in Figure \ref{fig:kirby}(a).) Note that $X_P$ is contractible if and only if the winding number of $P$ is $\pm 1$, in which case $Y_P$ is a homology sphere. Let $J_P \subset Y_P$ be the knot $S^1 \times \{\pt\}$. As an example, let $Q$ denote the \emph{Mazur pattern} shown in Figure \ref{fig:Q}, so called because $X_Q$ (with orientation reversed) is Mazur's original construction of a contractible $4$-manifold with boundary not homeomorphic to $S^3$ \cite{MazurContractible}.

\begin{figure}
\includegraphics[scale=0.75]{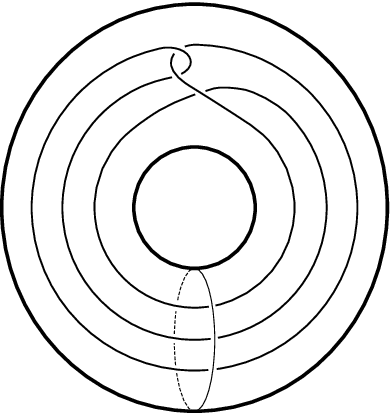}
\caption{The Mazur pattern knot $Q$ in the solid torus $S^1 \times D^2$.}
\label{fig:Q}
\end{figure}

\begin{proposition} \label{prop:PL}
If $P \subset S^1 \times D^2$ is a pattern knot with winding number $1$ that satisfies the conclusion of Theorem \ref{thm:nonsurjective}, then $J_P \subset Y_P$ does not bound a PL disk in any rational homology $4$-ball $X$ with $\partial X = Y_P$.
\end{proposition}

\begin{proof}
Suppose, toward a contradiction, that $J_P$ bounds a PL disk $\Delta \subset X$. We may assume that $\Delta$ is smooth away from finitely many singular points that are cones on knots $K_1, \dots, K_n$. By deleting neighborhoods of arcs in $\Delta$ connecting the cone points to $\partial \Delta$, we see that $J_P \conn {-K}$ is smoothly slice in $X$, where $K = K_1 \conn \dots \conn K_n$.

Let $W$ be obtained by attaching a $0$-framed $2$-handle to $X$ along $J_P \conn -K$. Then $W$ is a homology $S^2 \times D^2$, and a generator of $H_2(W)$ can be represented by a sphere $S$ with trivial normal bundle, obtained as the union of a slice disk for $J_P \conn {-K}$ and the core of the $2$-handle. As seen in Figure \ref{fig:kirby}(b-c), the boundary of $W$ is diffeomorphic to $0$-surgery on $P(-K)$; let $\gamma \subset S^3_0(P(-K))$ be the core of the surgery torus, represented in a surgery picture by a meridian of $P(-K)$.

Let $W'$ be obtained from $W$ by surgering out $S$; $W'$ is a homology $S^1 \times D^3$, and $H_1(W')$ is generated by $[\gamma]$. Attaching a $0$-framed $2$-handle to $W'$ along $\gamma$ produces a rational homology $4$-ball $Z$ whose boundary is $S^3$. The cocore of the new $2$-handle is a smooth slice disk for $P(-K)$, contradicting the conclusion of Theorem \ref{thm:main}.
\end{proof}


\begin{remark} \label{rmk:akbulut}
In \cite{AkbulutZeeman}, Akbulut proves that $J_Q$ does not bound a PL disk in $X_Q$. However, note that $Q$ does not satisfy the conclusion of Theorem \ref{thm:nonsurjective}, since $\tilde Q$ is the unknot. Moreover, if we form the Kirby diagram for $X_Q$ as in Figure \ref{fig:kirby}(a) using $Q$, we may interchange the dotted $1$-handle and $0$-framed $2$-handle to obtain a different contractible $4$-manifold $X'$ in which $J_Q$ is smoothly slice. Note that $X'$ is in fact diffeomorphic to $X$, but not rel boundary.
\end{remark}

\begin{remark} \label{rmk:matsumoto}
Given homology $3$-spheres $Y_0, Y_1$ and knots $K_0 \subset Y_0, K_1 \subset Y_1$, we say that $K_0$ and $K_1$ are \emph{homology concordant} if they cobound a smoothly embedded annulus in some homology cobordism $W$ between $Y_0$ and $Y_1$. Let $\hat \CC_\Z$ denote the group of homology concordance classes of knots in homology spheres that bound homology $4$-balls. Note that a knot $K \subset Y$ bounds a PL disk in some homology $4$-ball $X$ if and only if $K$ is homology concordant to some knot $K' \subset S^3$. Thus, Theorem \ref{thm:main} can be reformulated as saying that the natural inclusion $\CC_\Z \to \hat \CC_\Z$ is not surjective. This answers a question posed by Matsumoto \cite[Problem 1.31]{KirbyList}. (See also \cite[Proposition 6.3]{DavisRaySatellite}.)
\end{remark}

In order to prove Theorem \ref{thm:nonsurjective}, we recall two concordance invariants arising from the knot Floer complex of a knot $K \subset S^3$ \cite{OSzKnot, RasmussenThesis}. The invariant $\tau(K) \in \Z$, defined by Ozsv\'ath and Szab\'o \cite{OSz4Genus}, provides a lower bound for the smooth rational homology $4$-ball genus of $K$ (the minimum genus of a properly embedded surface in a rational homology $4$-ball with boundary $K$) and is additive under connected sum; as a result, it descends to a group homomorphism $\CC_\Q \to \Z$. More recently, Hom \cite{HomTau} defined an invariant $\epsilon(K) \in \{-1,0,1\}$, which together with $\tau(K)$ determines the value of $\tau$ for all cables of $K$. As a result, $\epsilon(K)=0$ whenever $K$ is smoothly $\Q$--slice.\footnote{In \cite{HomTau}, this result is only stated when $K$ is smoothly slice in $D^4$, but the same proof holds for any rational homology ball.}  The $\epsilon$ invariant is not a group homomorphism, but its behavior under connected sum is the same as that of the signs of real numbers under addition: positive plus positive equals positive, etc. (This property actually makes $\epsilon$ a rather powerful invariant; Hom has used it to find an infinite-rank direct summand of the group of topologically slice knots \cite{HomSummand}.)

Our technical main result, which immediately implies Theorem \ref{thm:nonsurjective}, is a formula for $\tau(Q(K))$ and $\epsilon(Q(K))$ in terms of $\tau(K)$ and $\epsilon(K)$, proved using bordered Heegaard Floer homology \cite{LOTBordered, LOTBimodules}:

\begin{theorem}\label{thm:Q}
Let $Q$ denote the Mazur pattern shown in Figure \ref{fig:Q}. For any knot $K \subset S^3$,
\begin{equation} \label{eq:tau(Q(K))-intro}
\tau(Q(K)) =
\begin{cases}
\tau(K) & \text{if } \tau(K) \le 0 \text{ and } \epsilon(K) \in \{0,1\} \\
\tau(K)+1 & \text{if } \tau(K) >0 \text{ or } \epsilon(K) = -1,
\end{cases}
\end{equation}
and
\begin{equation} \label{eq:epsilon(Q(K))-intro}
\epsilon(Q(K)) =
\begin{cases}
0 & \text{if } \tau(K) = \epsilon(K) = 0 \\
1 & \text{otherwise}.
\end{cases}
\end{equation}
In particular, $Q(K)$ is not rational homology concordant to any knot $L$ with $\epsilon(L)=-1$.
\end{theorem}

\begin{proof}[Proof of Theorem \ref{thm:main}]
If $L \subset S^3$ is any knot with $\epsilon(L)=-1$ (e.g., the left-handed trefoil), then $P = Q \conn {-L} \subset S^1 \times D^2$ satisfies the conclusion of Theorem \ref{thm:nonsurjective}. By Proposition \ref{prop:PL}, $J_L \subset \partial X_L$ does not bound a PL disk in any rational homology $4$-ball.

For the second part of the theorem, first note that for any pattern knot $P$ with winding number $1$, the Alexander polynomial of $J_P$ is equal to that of $\tilde P$. To see this, first note that the Kirby diagram in Figure \ref{fig:kirby}(a) presents $Y_P = \partial X_P$ as surgery on the two-component link $\tilde P \cup O \subset S^3$, where $O$ is the dotted unknot. Since $\lk(\tilde P, O) = 1$, we may find a Seifert surface $F$ for $\tilde P$ that meets $O$ in a single point. By puncturing $F$ at this point and capping it off in the surgery on $\tilde P$, we find a Seifert surface $F' \subset Y_P$ for $J_P$. It is easy to check that the Seifert forms of $F$ and $F'$ are equal, and hence $\Delta_{J_P} = \Delta_{\tilde P}$.

Now, if $P = Q \conn {-L}$ as above, then $\tilde P = -L$. In particular, if $L$ is a knot with $\Delta_L = 1$, then $\Delta_{J_P}=1$, and hence $J_P$ is topologically slice in $X_P$ by the famous theorem of Freedman and Quinn \cite[Theorem 11.7B]{FreedmanQuinn}. Hom \cite[Lemma 6.12]{HomComplex} showed that the negative, untwisted Whitehead double of the left-handed trefoil is an example of a knot $L$ with $\epsilon(L)=-1$ and $\Delta_L=1$, as required.
\end{proof}

Another consequence of Theorem \ref{thm:Q} is the following:

\begin{corollary} \label{cor:shrinking}
If $L \subset S^3$ is any knot with $\tau(L) > 0$, then $L$ is not rational homology cocordant to $Q^n(K)$ for any $K \subset S^3$ and $n > \tau(L)$.
\end{corollary}

\begin{proof}
Induction using Theorem \ref{thm:Q} shows that for any $n \ge 1$,
\begin{equation} \label{eq:iterate}
\tau(Q^n(K)) = \begin{cases}
\tau(K) & \text{if } \tau(K) \le 0 \text{ and } \epsilon(K) \in \{0,1\} \\
\tau(K) + 1 & \text{if } \tau(K) < 0 \text{ and } \epsilon(K) = -1 \\
\tau(K) + n & \text{if } \tau(K) = 0 \text{ and } \epsilon(K) = -1 \text{, or if } \tau(K) > 0.
\end{cases}
\end{equation}
In particular, for $n \ge 2$, $\tau(Q^n(K))$ cannot equal any number in $\{1, \dots, n-1\}$.
\end{proof}

Corollary \ref{cor:shrinking} implies that the images of the iterated satellite operators $Q^n$ (seen as functions on any of the groups $\CC$, $\CC_{\ex}$, $\CC_\Z$, or $\CC_\Q$) are strictly decreasing:
\begin{equation} \label{eq:shrinking}
\im(Q) \supsetneq \im(Q^2) \supsetneq \im(Q^3) \supsetneq \cdots.
\end{equation}
Cochran, Davis, and Ray \cite{CochranDavisRayInjectivity} showed that any pattern $P$ with winding number $\pm 1$ induces an injection on $\CC_{\Z}$, and any pattern $P$ with \emph{strong winding number $\pm 1$} (i.e., for which $\pi_1( \partial(S^1 \times D^2))$ normally generates $\pi_1(S^1 \times D^2 \minus \nbd(P))$) induces an injection on $\CC_{\ex}$. Moreover, Cochran and Harvey \cite{CochranHarveyGeometry} showed that such satellite operators are (quasi-)isometries with respect to certain natural metrics on each of the concordance groups. Note that the Mazur pattern $Q$ has strong winding number $1$ since $\tilde Q$ is the unknot \cite[Proposition 2.1] {CochranDavisRayInjectivity}. Therefore, \eqref{eq:shrinking} can perhaps be seen as an example of fractal structure in the concordance groups.

The operator $Q$ also appears in work of Cochran, Franklin, Hedden, and Horn \cite{CochranFranklinHeddenHorn}, who used it to give the first known examples of non-concordant knots whose $0$-surgeries are homology cobordant rel meridians. Specifically, they showed that the $0$-surgeries on $K$ and $Q(K)$ are homology cobordant rel meridians for any knot $K$, while there exist knots for which $\tau(K) \ne \tau(Q(K))$. Ray \cite{RayIterates} extended this argument to show that for such knots, all of the iterates $Q^n(K)$ $(n \in \N)$ are distinct in concordance. The formula for $\tau(Q(K))$ given above confirms and strengthens both of these results. (See Remark \ref{rmk:CFHH} for more details.)

We assume throughout the paper that the reader is familiar with knot Floer homology \cite{OSzKnot}, sutured Floer homology \cite{JuhaszSutured}, and bordered Heegaard Floer homology \cite{LOTBordered, LOTBimodules}. (For a quick summary of the latter, see the author's exposition in \cite[Section 2]{LevineDoublingOperators}.) All Floer homology groups are taken with coefficients in $\F = \Z/2\Z$. In Section \ref{sec:alexander}, we discuss the role of relative spin$^c$ structures in the bordered theory, with an eye toward computations of knot Floer homology for satellite knots. In Section \ref{sec:2bridge}, we compute the bordered Floer homology of $S^1 \times D^2 \minus \nbd(Q)$, making use of Lipshitz, Ozsv\'ath, and Thurston's arc-slides algorithm \cite{LOTFactoring}, as implemented in Python by Bohua Zhan \cite{ZhanComputations}. We then use this computation to determine the values of $\tau$ for $Q(K)$ (Section \ref{sec:tau}) and for the $(2,1)$ and $(2,-1)$ cables of $Q(K)$ (in Section \ref{sec:epsilon}), and finally deduce $\epsilon(Q(K))$ using Hom's formula for $\tau$ of cables \cite{HomTau}, leading to the proof of Theorem \ref{thm:Q}.

\subsection*{Acknowledgments}

The author is deeply grateful to Bohua Zhan for making available his remarkable software package for computing Heegaard Floer homology; to Matthew Hedden, Jen Hom, Robert Lipshitz, Peter Ozsv\'ath, Mark Powell, Vinicius Ramos, Arunima Ray, and Danny Ruberman for many interesting conversations; and to the referee for helpful suggestions.

\section{Alexander gradings in bordered Floer homology} \label{sec:alexander}

In this section, we elaborate on the pairing theorem for knot Floer homology given by Lipshitz, Ozsv\'ath, and Thurston \cite[Theorem 11.21]{LOTBordered}. Specifically, we will show that bordered Floer homology determines the absolute Alexander grading on the knot Floer homology of a knot in a manifold obtained by gluing, not just the relative Alexander grading as was originally stated. The most important result is Proposition \ref{prop:Aglue}, which provides a useful technique for computing Alexander gradings in the knot Floer homology of satellite knots. (See Remark \ref{rmk:HomPetkova} regarding other strategies for such computations used by Petkova \cite{PetkovaCables}, Hom \cite{HomTau}, and the author \cite{LevineDoublingOperators}.)

\subsection{The knot Floer complex and \texorpdfstring{$\tau(K)$}{\texttau(K)}} \label{sec:HFK}

We begin by recalling some basics concerning knot Floer homology. For simplicity, suppose that $K$ is a knot in a homology sphere $Y$. In this discussion, we shall use the convention for relative spin$^c$ structures used in sutured Floer homology \cite{JuhaszSutured}. Specifically, let $X_K = Y \minus \nbd(K)$, equipped with a pair of meridional sutures $\Gamma$ on the boundary. We fix a vector field $\vec v$ along $\partial X_K$ that points into $X_K$ along $R_-(\Gamma)$, parallel to $\partial X_K$ and transverse to the sutures along $\Gamma$, and out of $X_K$ along $R_+(\Gamma)$. A \emph{relative spin$^c$ structure} is a homology class of nonvanishing vector fields on $X_K$ that restrict to $\vec v$ on $\partial X_K$. (Here, two vector fields on $X_K$ are \emph{homologous} if they are homotopic outside of finitely many balls in $X_K$; the homotopies are required to be fixed on $\partial X_K$.) The set of relative spin$^c$ structures is denote $\ul\Spin^c(Y,K)$ and is an affine space for $H^2(X_K, \partial X_K)$, which by excision is isomorphic to $H^2(Y,K)$. Let $\PD[\mu] \in H^2(X_K, \partial X_K)$ denote the Poincar\'e--Lefschetz dual of the meridian of $K$.\footnote{A different convention also appears in the literature: specifically, Ozsv\'ath and Szab\'o \cite{OSzRational} define $\vec v$ to point outward on all of $\partial X_K$. The only difference between the two conventions is the formula for the Alexander grading in terms of the first Chern class; using the other convention,
the right side of \eqref{eq:absalex} should have an additional term of $-\frac12 \gen{\PD[\mu], [F]}$. Our convention agrees with the convention for spin$^c$ structures in sutured Floer homology \cite{JuhaszSutured}; it seems to behave more naturally with respect to bordered constructions.}

Given a doubly-pointed Heegaard diagram $(\HH, z, w)$ presenting $K$, the \emph{knot Floer complex} $\CFKm(\HH)$ is freely generated over $\F[U]$ by points $\x \in \T_\alpha \cap \T_\beta$, with differential given by
\[
\partial(\x) = \sum_{\y \in \mathfrak{S}(\HH)} \sum_{\substack{\phi \in \pi_2(\x,\y) \\ \mu(\phi)=1 }} \# (\hat\MM(\phi)) \, U^{n_w(\phi)} \y.
\]
Each generator $\x$ has an associated relative spin$^c$ structure in $\ul\Spin^c(Y,K)$, denoted $\srel_{w,z}(\x)$; for any $\phi \in \pi_2(\x,\y)$, we have
\begin{equation} \label{eq:srel-shift-closed}
\srel_{w,z}(\x) - \srel_{w,z}(\y) = (n_z(\phi) - n_w(\phi)) \PD[\mu]
\end{equation}
\cite[Lemma 2.5]{OSzKnot}. The \emph{Alexander grading} of a generator is defined as
\begin{equation} \label{eq:absalex}
A(\x) = \frac12 \gen{c_1(\srel_{w,z}(\x)), [F]},
\end{equation}
where $[F] \in H_2(X_K, \partial X_K)$ denotes the homology class of a Seifert surface for $K$. We extend this grading to $\CFKm(\HH)$ by setting $A(U^n \cdot \x) = A(\x) - n$; \eqref{eq:srel-shift-closed} shows that $A$ determines a filtration on $\CFKm(\HH)$. (Alternately, we may define $\srel_{w,z}(U^n \cdot \x) = \srel_{w,z}(\x) - n \PD[\mu]$, and simply view $\CFKm(\HH)$ as being filtered by the set of relative spin$^c$ structures, ordered by the action of $\PD[\mu]$.)

The filtered chain homotopy type of $\CFKm(\HH)$ is an invariant of the isotopy class of $K$; any complex of this homotopy type will be denoted $\CFKm(Y,K)$, and the grading by spin$^c$ structures is denoted $\srel_{Y,K}$ (in the absence of an actual Heegaard diagram). The associated graded complex of $\CFKm(Y,K)$ is denoted $\gCFKm(Y,K)$; the differential counts disks that avoid the basepoint $z$. The homology of $\gCFKm(Y,K)$ is denoted $\HFKm(Y,K)$; as an $\F$--vector space, it decomposes by Alexander grading as
\[
\HFKm(Y,K) = \bigoplus_{s \in \Z} \HFKm(Y,K,s),
\]
with the action of $U$ taking $\HFKm(Y,K,s)$ to $\HFKm(Y,K,s-1)$.

For any knot $K \subset S^3$, $\HFKm(S^3,K)$ is (non-canonically) isomorphic to the direct sum of $\F[U]$ and a finitely-generated, torsion $\F[U]$--module. The invariant $\tau(K)$ is equal to
\begin{equation}
\tau(K) = - \max\{ s \mid U^n \cdot \HFKm(S^3,K,s) \ne 0 \text{ for all } n \ge 0\}.
\end{equation}
In other words, $\tau(K)$ is minus the Alexander grading of $1 \in \F[U] \subset \HFKm(S^3,K)$. (See \cite[Lemma A.2]{OSzThurstonLegendrian} for the proof that this agrees with the original definition of $\tau$ in terms of the filtration on $\CFK(S^3,K)$.) Ozsv\'ath and Szab\'o proved that if $X$ is a rational homology $4$-ball with boundary $S^3$, and $K$ is the boundary of a smoothly embedded surface in $X$ of genus $g$, then $\abs{\tau(K)} \le g$ \cite{OSz4Genus}. Since $\tau$ is additive under connected sums, it descends to a homomorphism $\CC_\Q \to \Z$.

\subsection{Relative \texorpdfstring{spin$^c$}{spin-c} structures on bordered manifolds} \label{sec:vector}

We now turn to bordered Floer homology \cite{LOTBordered}. Let $\ZZ = (Z, \a, M, z)$ be a pointed matched circle of genus $k$. (Here, $Z$ is an oriented circle, $\a$ is a set of $4k$ points in $Z$, $M\co \a \to \{1, \dots, 2k\}$ is a two-to-one function, and $z \in Z \minus \a$.) Let $F(\ZZ)$ denote the surface associated to $\ZZ$. The surface $F(\ZZ)$ admits a handle decomposition with a single $0$-handle $\Delta$ whose boundary is identified with $Z$; $2k$ $1$-handles whose feet are at the points of $\a$, paired according to $M$; and a single $2$-handle. Let $\AA(\ZZ)$ denote the bordered algebra associated to $\ZZ$.

We shall make use of Huang and Ramos's construction of a topological grading on bordered Floer homology \cite{HuangRamosBordered}. In this discussion, we refer to a smooth section of the bundle $TF(\ZZ) \oplus \ul\R$ as a \emph{vector field along $F(\ZZ)$}, where $\ul\R$ is a trivial real line bundle over $F(\ZZ)$ equipped with a choice of orientation. For any $k$-element subset $\s \subset \{1,\dots, 2k\}$, we fix a nonvanishing vector field $\vec v_\s$ along $F(\ZZ)$ according to the construction given in \cite[Definition 2.1]{HuangRamosBordered}.
For subsets $\s,\t \subset \{1,\dots, 2k\}$ of order $k$, let $\GG(\ZZ, \s, \t)$ denote the set of homotopy classes of vector fields on $F(\ZZ) \times I$ restricting to $\s$ on $F(\ZZ) \times \{0\}$ and to $\t$ on $F(\ZZ) \times \{1\}$, and let $\ul\Spin^c(\ZZ, \s, \t)$ denote the set of homology classes of such vector fields. The latter is an affine set for $H^2(F(\ZZ) \times I, F(\ZZ) \times \partial I) \cong H_1(F(\ZZ) \times I) \cong H_1(F(\ZZ))$. There is a free action of $\Z$ on $\GG(\ZZ, \s, \t)$, where the action of $n \in \Z$ is denoted $[\vec v] \mapsto \lambda^n \cdot [\vec v]$, whose quotient map is precisely the forgetful map $\Phi_{\s,\t}\co \GG(\ZZ, \s, \t) \to \ul\Spin^c(\ZZ, \s, \t)$. Let
\[
\GG(\ZZ) = \coprod_{\substack{\s, \t \subset \{1, \dots, 2k\} \\ \abs{\s} = \abs{t}=k}} \GG(\ZZ, \s, \t)
\quad \text{and} \quad
\ul\Spin^c(\ZZ) = \coprod_{\substack{\s, \t \subset \{1, \dots, 2k\} \\ \abs{\s} = \abs{t}=k}} \ul\Spin^c(\ZZ, \s, \t),
\]
each of which equipped with a groupoid structure in which multiplication is given by concatenation in the $I$ factor; combine the forgetful maps $\Phi_{\s,\t}$ into a single map $\Phi$.\footnote{Huang and Ramos refer to the former as $G(\ZZ)$, but we prefer the notation $\GG(\ZZ)$ to avoid confusion with the grading group given by Lipshitz, Ozsv\'ath, and Thurston \cite[Section 3.3.2]{LOTBordered}.} Huang and Ramos define a grading $\gr$ on $\AA(\ZZ)$ taking values in $\GG(\ZZ)$; they also show that there is a homomorphism $\FF \co \GG(\ZZ) \to G'(\ZZ)$, where $G'(\ZZ)$ is the grading group from \cite[Section 3.3.1]{LOTBordered}, under which their grading agrees with the original grading on $\AA(\ZZ)$. For $a \in \AA(\ZZ)$, let $\srel(a) = \Phi(\gr(a)) \in \ul\Spin^c(\ZZ)$.

Next, let $Y$ be a bordered $3$-manifold with boundary parametrized by $F(\ZZ)$. We identify $TY|_{F(\ZZ)}$ with $TF(\ZZ) \oplus \ul\R$, where the outward normal is mapped to the positive $\ul\R$ direction. For each $k$-element subset $\s \subset \{1, \dots, 2k\}$, let $\Vect(Y,\s)$ and $\ul\Spin^c(Y, \s)$ denote the set of homotopy classes and homology classes, respectively, of nonvanishing vector fields on $Y$ restricting to $v_{\s}$. Elements of $\ul\Spin^c(Y)$ are called \emph{relative spin$^c$ structures} (relative to $\vec v_{\s}$); note that $\ul\Spin^c(Y,\s)$ is an affine set for $H^2(Y,\partial Y;\Z)$. Let $\Phi_\s \co \Vect(Y,\s) \to \ul\Spin^c(Y, \s)$ denote the forgetful map. Once again, there is an action of $\Z$ on $\Vect(Y,\s)$ (not necessarily free) whose quotient map is precisely $\Phi_\s$. Define
\[
\SS(Y) = \coprod_{\substack{\s \subset \{1, \dots, 2k\} \\ \abs{\s}=k}} \Vect(Y,\s) \quad \text{and} \quad
\ul\Spin^c(Y) = \coprod_{\substack{\s \subset \{1, \dots, 2k\} \\ \abs{\s}=k}} \Spin^c(Y,\s),
\]
and combine the maps $\Phi_\s$ into a single map $\Phi$.
The groupoid $\GG(\ZZ)$ acts on $\SS(Y)$ from the right by concatenation, and the action descends to an action of $\ul\Spin^c(\ZZ)$ on $\ul\Spin^c(Y)$. (In a similar manner, if $\partial Y$ is parametrized by $-F(\ZZ)$, then $\GG(\ZZ)$ and $\ul\Spin^c(\ZZ)$ act on $\SS(Y)$ and $\ul\Spin^c(Y)$ from the left.)

Let $\HH$ be a bordered Heegaard diagram for $Y$. For each generator $\x \in \mathfrak{S}(\HH)$, let $o(\x)$ be the $k$-element subset of $\{1,\dots,2n\}$ corresponding to the arcs occupied by $\x$. Lipshitz, Ozsv\'ath, and Thurston \cite[Section 4.3]{LOTBordered} construct a nowhere-vanishing vector field $\vec v_z(\x)$ on $Y$ whose restriction to a collar neighborhood of $\partial Y$ is $\vec v_{o(\x)}$, and define the relative spin$^c$ structure associated to $\x$, denoted $\srel_z(\x) \in \ul\Spin^c(Y)$, to be the homology classes of $v_{\x}$. Subsequently, Huang and Ramos defined $\gr(\x) \in S(Y)$ to be the homotopy class of $\vec v_z(\x)$, and proved that this assignment determines a grading on the bordered invariants $\CFA(\HH)$ and $\CFD(\HH)$ that is compatible with the algebraic structures of those invariants. That is, if in $\CFA(\HH)$ the generator $\y$ appears in $m_{k+1}(\x, a_1, \cdots, a_k)$, then
\begin{equation} \label{eq:gr-shift}
\gr(\y) = \lambda^{k-1} \gr(\x) \cdot \gr(a_1)\cdots \gr(a_k).
\end{equation}
It follows that
\begin{equation} \label{eq:srel-shift}
\srel_z(\y) = \srel_z(\x) \cdot \srel(a_1) \cdots \srel(a_k).
\end{equation}
A similar statement holds for $\CFD$; see \cite[Theorem 1.3]{HuangRamosBordered}.

Moreover, the maps used in \cite{LOTBordered} to prove the invariance of bordered Heegaard Floer homology are in fact graded chain homotopy equivalences. This is not spelled out explicitly in \cite{HuangRamosBordered}, but it follows along the same lines as the proof of invariance for Huang and Ramos's earlier work \cite{HuangRamosHF}, modified for the bordered setting. It follows that the \emph{graded} chain homotopy type of $\CFA(\HH)$ or $\CFD(\HH)$, where the grading has values in $\SS(Y)$, is an invariant of $Y$; by abuse of notation, we refer to any $\AA_\infty$-module or type-$D$ structure with this graded chain homotopy type as $\CFA(Y)$ or $\CFD(Y)$, respectively. Note also that the chain homotopy equivalences used in the ``edge reduction'' procedure for simplifying a chain complex, $\AA_\infty$-module, or type-$D$ structure (see \cite[Section 2.6]{LevineDoublingOperators}) are graded.

For the present purposes, the upshot of this discussion is that each homogeneous generator $x$ of (a module representing) $\CFA(Y)$ or $\CFD(Y)$ has an associated relative spin$^c$ structure, denoted $\srel_Y(x)$, which is obtained by applying the forgetful map $\Phi$ to $\gr(x)$. This is true even when we are working with a representative for $\CFA(Y)$ or $\CFD(Y)$ that is not actually the complex associated to a Heegaard diagram, a fact that was not fully spelled out in \cite{LOTBordered}. (We do not not need to make use of the Maslov component of the grading in this paper.)

\subsection{Knots in bordered 3-manifolds} \label{sec:borderedknot}

Next, we consider knots in bordered manifolds. As explained in \cite[Section 11.4]{LOTBordered}, a bordered Heegaard diagram $(\HH,z)$ for $Y$ together with a second basepoint $w$ in the interior of the Heegaard surface determines a knot $K \subset Y$, a segment of which lies in $\partial Y$. (We refer to $K$ as a \emph{based knot}.) To be precise, fix a Riemannian metric $g$ and a self-indexing Morse function $f$ on $Y$ that are compatible with the Heegaard diagram $\HH$ (in the sense of \cite[Section 4.8]{LOTBordered}). The basepoints $z$ and $w$ each determine flowlines $\gamma_z$ and $\gamma_w$ connecting the unique index-$0$ and index-$3$ critical points of $f$; note that $\gamma_z \subset \partial Y$. If we orient each of these flowlines upward, the knot $K$ is defined to be $\gamma_w - \gamma_z$. (For an alternate description, let $t_\alpha$ be an arc from $z$ to $w$ in the complement of the $\alpha$ curves, and let $t_\beta$ be an arc from $w$ to $z$ in the complement of the $\beta$ curves; we obtain $K$ by pushing $t_\alpha$ into the $\alpha$ handlebody and $t_\beta$ into the $\beta$ handlebody.)

Lipshitz, Ozsv\'ath, and Thurston define type $A$ and $D$ modules $\CFAm(\HH, z,w)$ and $\CFDm(\HH, z, w)$ over the ground ring $\F[U]$, where the original definitions of the differentials on $\CFA(\HH,z)$ and $\CFD(\HH,z)$ are modified so that a holomorphic disk with multiplicity $m$ at $w$ contributes a factor of $U^m$.
That is, in $\CFAm(\HH,z,w)$, the $\AA_\infty$ multiplications are given by
\begin{equation} \label{eq:CFA-def}
m_{n+1}(\x, a(\bm\rho_1), \dots, a(\bm\rho_n)) = \sum_{\y \in \mathfrak{S}(\HH)}
\sum_{\substack{B \in \pi_2(\x,\y) \\ \ind(B, \vec{\bm\rho}) =1}} \# \MM^B( \x, \y; \bm\rho_1, \dots, \bm\rho_n)
\, U^{n_w(B)} \y,
\end{equation}
where all notation is as in \cite[Section 7]{LOTBordered}.
The chain homotopy types of $\CFAm(\HH,z,w)$ and $\CFDm(\HH,z,w)$ are invariants of the isotopy class of $K$, where a segment of $K$ is constrained to lie in $\partial Y$ throughout the isotopy. We shall explain how to construct relative spin$^c$ gradings on these invariants. We focus on $\CFAm$; the case of $\CFDm$ is similar.

The complement $Y_K$ of a regular neighborhood of $K$ is naturally a \emph{bordered sutured manifold}, in the sense of Zarev \cite[Definition 3.5]{ZarevBordered}. Specifically, let $F'(\ZZ)$ be $F(\ZZ)$ minus its $0$- and $2$-handles plus an an annulus connecting the two boundary circles, with a pair of parallel sutures $\Gamma$ contained in this annulus. The sutures divide $F'(\ZZ)$ into regions $R_+$ and $R_-$ with $\chi(R_+)=0$ and $\chi(R_-) = -2k$; the pointed matched circle $\ZZ$ determines a parametrization of $R_-$. The boundary of $\partial Y_K$ is then naturally identified with $F'(\ZZ)$. Moreover, $Y_K$ is represented by the bordered sutured Heegaard diagram $\HH'$ obtained from $\HH$ by deleting a neighborhood of $w$. Note that the generating sets $\mathfrak{S}(\HH)$ and $\mathfrak{S}(\HH')$ are the same. Moreover, the bordered sutured invariant $\widehat{BSA}(\HH')$ defined by Zarev is precisely equal to the quotient $\CFA(\HH,z,w) = \CFAm(\HH,z,w)/(U=0)$.

As noted by Huang and Ramos \cite[Remark 1.5]{HuangRamosBordered}, the discussion from the previous section carries through for bordered sutured manifolds. Just as in the absolute case, for each $k$-element subset $\s \subset \{1, \dots, 2k\}$, we fix a nonvanishing vector field $\vec v_\s'$ along $F'(\ZZ)$; define groupoids of homology and homotopy classes of vector fields on $F'(\ZZ) \times I$, denoted $\tilde\GG(\ZZ)$ and $\tilde{\ul\Spin}{}^c(\ZZ)$, analogous to the constructions in the previous section. The algebra $\AA(\ZZ)$ has a grading $\tilde{\gr}$ valued in $\tilde \GG(\ZZ)$; the image of $\tilde\gr(a)$ in $\tilde{\ul\Spin}{}^c(Y)$ is denoted $\tilde \srel(a)$. Let $\Vect(Y,K,\s)$ and $\ul\Spin^c(Y,K,\s)$ be the sets of homotopy classes and homology classes, respectively, of vector fields on $Y_K$ extending $\vec v_{\s}'$, and let
\[
\SS(Y,K) = \coprod_{\s} \Vect(Y,K,\s) \quad \text{and} \quad \ul\Spin^c(Y,K) = \coprod_{\s} \ul\Spin^c(Y,K,\s);
\]
these admit actions by $\tilde\GG(\ZZ)$ and $\tilde{\ul\Spin}{}^c(\ZZ)$, respectively. Then $\CFA(\HH,z,w)$ has a grading $\tilde{\gr}$ valued in $\SS(Y,K)$; the image of $\widetilde{\gr}(\x)$ in $\ul\Spin^c(Y,K)$ is denoted $\srel_{w,z}(\x)$.

We now show that $\CFAm(\HH,z,w)$ is also graded by relative spin$^c$ structures. For each $\s$, $\ul\Spin^c(Y,K,\s)$ is an affine set for $H^2(Y_K, \partial Y_K)$. Let $\PD[\mu] \in H^2(Y, \nbd(K)) \cong H^2(Y_K, \partial Y_K)$ denote the Poincar\'e--Lefschetz dual of the meridian $[\mu] \in H_1(Y_K;\Z)$, which generates the kernel of the restriction map $H^2(Y,\nbd(K)) \to H^2(Y)$. There are maps
\[
\Gamma_\s \co \ul\Spin^c(Y,K,\s) \to \ul\Spin^c(Y,\s),
\]
given by extending the vector fields over $\nbd(K)$ in a canonical way, whose fibers are the orbits of the action of $\PD[\mu]$; the maps $\Gamma_\s$ combine to give a single map
\[
\Gamma \co \ul\Spin^c(Y,K) \to \ul\Spin^c(Y),
\]
satisfying $\Gamma(\srel_{w,z}(\x)) = \srel_z(\x)$. (See \cite[Section 2.2]{OSzRational} for more details of the analogous construction for knots in closed manifolds.) If $B \in \pi_2(\x,\y)$ is a class (in $\HH$) with associated algebra elements $a_1, \dots, a_n$, then
\[
\Gamma(\srel_{w,z})(\y) = \Gamma( \srel_{w,z}(\x) \cdot \tilde{\srel}(a_1) \cdots \tilde{\srel}(a_n) )
\]
by \eqref{eq:srel-shift}, so $\srel_{w,z}(\y)$ and $\srel_{w,z}(\x) \cdot \tilde{\srel}(a_1) \cdots \tilde{\srel}(a_n)$ must differ by an element of $\PD[\mu]$; more precisely, a bordered analogue of \cite[Lemma 2.5]{OSzKnot} says that
\[
\srel_{w,z}(\x) \cdot \tilde{\srel}(a_1) \cdots \tilde{\srel} (a_n) - \srel_{w,z}(\x) = -n_w(B) \PD[\mu].
\]
Therefore, if we define the relative spin$^c$ grading on $\CFAm(\HH,z,w)$ by
\[
\srel_{w,z}(U^n \cdot \x) = \srel_{w,z}(\x) - n \PD[\mu],
\]
the $\AA_\infty$ multiplications respect the spin$^c$ structures just as in \eqref{eq:srel-shift}. The graded chain homotopy type of $\CFAm(\HH,z,w)$ is an invariant of the knot $K$ (once again, under isotopies leaving a segment of $K$ fixed on $\partial Y$). We denote any representative of this homotopy type by $\CFAm(Y,K)$, and refer to its spin$^c$ grading as $\srel_{Y,K}$.

We may now state the graded version of the pairing theorem for knot Floer homology. First, suppose $Y$ is a homology $3$-sphere with a bordered decomposition $Y_1 \cup_{F(\ZZ)} Y_2$, and $K$ is a based knot in $Y_1$, which may be viewed as a knot in $Y$. There is a gluing map
\[
\Psi \co \coprod_{\substack{\s \subset \{1, \dots, 2k\} \\ \abs{\s} =k }} \left( \ul\Spin^c(Y_1,K,\s) \times \ul\Spin^c(Y_2,\s) \right) \to \ul\Spin^c(Y,K).
\]
The version of the pairing theorem that we shall use states:
\begin{theorem} \label{thm:pairing}
There is a homotopy equivalence
\[
\gCFKm(Y,K) \simeq \CFAm(Y_1,K) \boxtimes \CFD(Y_2)
\]
that respects the grading by relative spin$^c$ structures, in the sense that for homogeneous elements $x_1 \in \CFAm(Y_1,K)$ and $x_2 \in \CFD(Y_2)$ whose idempotents agree, we have
\begin{equation} \label{eq:srel-gluing}
\srel_{Y,K}(x_1 \otimes x_2) = \Psi(\srel_{Y_1,K}(x_1), \srel_{Y_2}(x_2)).
\end{equation}
\end{theorem}

\begin{proof}
The existence of the homotopy equivalence is simply \cite[Theorem 11.21]{LOTBordered}, and \eqref{eq:srel-gluing} follows directly from the construction.
\end{proof}

Thus, the spin$^c$ grading on bordered Floer homology can be used to recover the \emph{absolute} Alexander grading on $\HFKm$, not just the relative grading (as stated in \cite{LOTBordered}). Moreover, similar pairing theorems also apply for computations using bimodules, all of which respect the grading by relative spin$^c$ structures.

\subsection{Satellite knots} \label{sec:satellite}

We now give a more concrete description of the way that bordered Floer homology determines the Alexander gradings on $\HFKm$ of satellite knots.

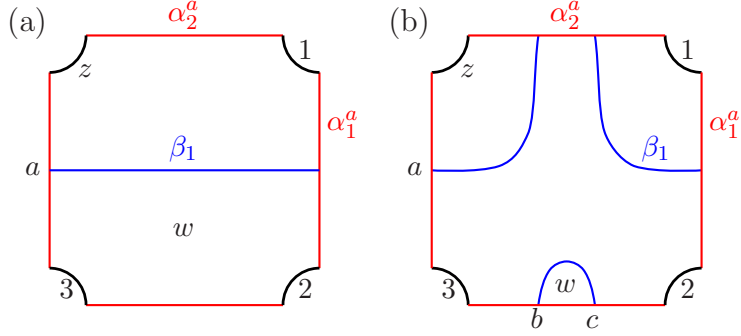
\begin{figure}
\labellist
 \pinlabel (a) at 0 115
 \pinlabel (b) at 144 115
 \small
 \pinlabel $z$ at 22 97
 \pinlabel $w$ at 59 37
 \pinlabel $1$ at 105 105
 \pinlabel $2$ at 105 15
 \pinlabel $3$ at 15 15
 \pinlabel $a$ [r] at 8 60
 \pinlabel {{\color{red} $\alpha_1^a$}} [l] at 110 77
 \pinlabel {{\color{red} $\alpha_2^a$}} [b] at 59 110
 \pinlabel {{\color{blue} $\beta_1$}} [b] at 59 60
 \pinlabel $z$ at 166 97
 \pinlabel $w$ at 203 16
 \pinlabel $1$ at 249 105
 \pinlabel $2$ at 249 15
 \pinlabel $3$ at 159 15
 \pinlabel $a$ [r] at 152 60
 \pinlabel $b$ [b] at 192 -3
 \pinlabel $c$ [b] at 213 -3
 \pinlabel {{\color{red} $\alpha_1^a$}} [l] at 254 77
 \pinlabel {{\color{red} $\alpha_2^a$}} [b] at 203 110
 \pinlabel {{\color{blue} $\beta_1$}} [b] at 237 60
 \endlabellist
\includegraphics{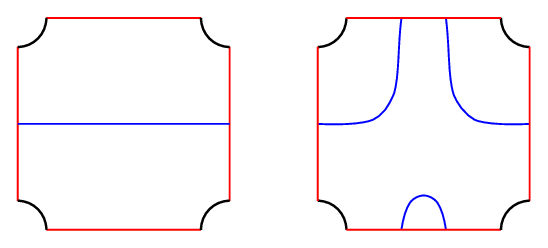}
\caption{Doubly-pointed bordered Heegaard diagrams for the knots $C$ (a) and $C_{2,1}$ (b) in the solid torus $V$. The boundary segments are labeled according to the convention for $\CFA$.}
\label{fig:cables}
\end{figure}

Let $V$ denote the solid torus $S^1 \times D^2$, equipped with the standard bordered structure described in \cite[Section 11.4]{LOTBordered}. That is, in any bordered Heegaard diagram for $V$, $\alpha_1^a$ represents a meridian $\mu_V = \{\pt\} \times \partial D^2$ and $\alpha_2^a$ represents a longitude $\lambda_V = S^1 \times \{\pt\}$. Let $P \subset V$ be a based knot in $V$, represented by a doubly-pointed bordered Heegaard diagram $(\HH, z, w)$ as above. Label the boundary regions of $\HH$ $R_0, R_1, R_2, R_3$ according to the conventions for $\CFA$. Two specific examples that will be useful below are represented by the genus-$1$ Heegaard diagrams in Figure \ref{fig:cables}: let $C$ be a copy of $S^1 \times \{\pt\}$, and let $C_{2,1}$ be a curve in $\partial V$ representing the homology class $2 \lambda_V + \mu_V$ (a $(2,1)$ curve).

Let $F_\mu = \{\pt\} \times D^2 \subset V$; the homology class of $F_\mu$ generates $H_2(V, \partial V)$. There is a periodic domain $\PP_\mu$ representing $[F_\mu]$, whose multiplicities in the four boundary regions are $0,0,1,1$, respectively. Set $m = n_w(\PP_\mu)$; this is the \emph{winding number} of $P$. Since $P$ is homologous in $V$ to $m \lambda_V$, there is an oriented surface $F_\lambda \subset V$ whose boundary is the union of $P$ and $m$ parallel copies of $-\lambda_V$.

For any knot $K \subset S^3$, let $X_K$ denote the exterior of $K$ equipped with the bordered structure given by the $0$-framing. Let $F_K$ be a Seifert surface for $K$, which represents a generator of $H_2(X_K, \partial X_K)$. The type-$D$ structure $\CFD(X_K)$ splits as a direct sum (of $\F$-vector spaces) $V_0 \oplus V_1$ corresponding to the two idempotents $\iota_0, \iota_1 \in \AA(T^2)$. (That is, $V_i = \iota_i \cdot \CFD(X_K)$.) By \cite[Proposition 11.19]{LOTBordered}, $V_0$ (together with its internal differential $D$) is chain homotopy equivalent to $\gCFK(S^3,K)$. Indeed, given a Heegaard diagram $\HH$ for $K$, $(V_0,D)$ is actually isomorphic to $\gCFK(\HH')$, where $\HH'$ is the complex obtained by gluing $\HH$ to the Heegaard diagram for $(V,C)$ in Figure \ref{fig:cables}(a). Moreover, this identification respects the gradings by relative spin$^c$ structures. Thus, $V_0$ admits an Alexander grading, which we shall denote by $A_K$. (In a similar manner, $V_1$ can be identified with the longitude Floer complex of $K$ \cite{EftekharyWhitehead}, which likewise admits an Alexander grading.)

When we form the union $S^3 = V \cup X_K$, the knot $P \subset V$ becomes the satellite $P(K)$. According to Theorem \ref{thm:pairing}, there is a chain homotopy equivalence
\begin{equation} \label{eq:gCFKm(S3,P(K))}
\gCFKm(S^3, P(K)) \simeq \CFAm(V,P) \boxtimes \CFD(X_K).
\end{equation}
Moreover, this identification determines the grading of $\gCFKm(S^3,P(K))$ by relative spin$^c$ structures, and thus the absolute Alexander grading, which we denote by $A_{P(K)}$.

A key tool that we will use in our computations in Sections \ref{sec:tau} and \ref{sec:epsilon} is the following:

\begin{proposition} \label{prop:Aglue}
Let $P \subset V$ be a based knot with winding number $m$. For each element $a \in \CFAm(V,P) \cdot \iota_0$ that is homogeneous with respect to the spin$^c$ grading, there exists a constant $C_a$ with the following property: For any knot $K \subset S^3$, and any homogeneous element $x \in \iota_0 \CFD(X_K)$, we have
\begin{equation} \label{eq:Aglue}
A_{P(K)}(a \otimes x) = m A_K(x) + C_a.
\end{equation}
\end{proposition}

\begin{proof}
We may construct a Seifert surface $G$ for $P(K)$ as the union of $F_\lambda \subset V$ with $\abs{m}$ parallel copies of a Seifert surface $F_K$ for $K$. (If $m$ is negative, we take these copies of $F_K$ with reversed orientation.) The relative spin$^c$ structures $\srel(a) \in \ul\Spin^c(V, \partial V \cup K)$ and $\srel(x) \in \ul\Spin^c(X_K, \partial X_K)$ glue together to give a relative spin$^c$ structure $\srel(a \otimes x) \in \Spin^c(S^3, P(K))$. We then have:
\begin{align} \label{eq:Aglue-proof}
A_{P(K)}(a \otimes x) &= \frac12 \gen{c_1(\srel_{S^3,K}(a \otimes x)), [G]} \\
\nonumber &= \frac12 \gen{c_1(\srel_{V,P}(a)), [F_\lambda]} + \frac{m}{2} \gen{c_1(\srel_{X_K}(x)), [F_K]} \\
\nonumber &= \frac12 \gen{c_1(\srel_{V,P}(a)), [F_\lambda]} + m A_K(x).
\end{align}
Thus, we define $C_a = \frac12 \gen{c_1(\srel_{V,P}(a)), [F_\lambda]}$, which depends only on $a$ and not on the choice of $K$.
\end{proof}

The value of Proposition \ref{prop:Aglue} is that it enables us to compute Alexander gradings without using a Heegaard diagram. Specifically, in Section \ref{sec:tau} we will compute $\CFAm(V,Q)$, where $Q$ is the Mazur pattern knot in Figure \ref{fig:Q}, without keeping track of the relative spin$^c$ structures associated to the various elements. Since applying the satellite operation $Q$ to the unknot $O$ yields the unknot, $\CFAm(V,Q) \boxtimes \CFD(X_O)$ computes $\HFKm(S^3,O)$, which is simply $\F[U]$ generated by an element in Alexander grading $0$. This enables us to determine the constants $C_a$ associated to some generators $a \in \CFAm(V,Q)$. We can then use Proposition \ref{prop:Aglue} to determine the absolute Alexander gradings of the relevant generators of $\CFKm(S^3,Q(K))$ for any knot $K$, and this computation suffices to determine $\tau(Q(K))$. The same reasoning is used in Section \ref{sec:epsilon} to study $\CFAm(V,Q_{2,1})$, where $Q_{2,1}$ denotes the $(2,1)$ cable of $Q$.

\begin{remark} \label{rmk:HomPetkova}
Proposition \ref{prop:Aglue} is closely related to the strategy used by Hom in \cite[Section 4]{HomTau} for computing $\tau$ for cable knots. Given a doubly-pointed Heegaard diagram for a knot $K \subset S^3$, Ozsv\'ath and Szab\'o give an explicit formula for the Alexander grading of each generator in terms of topological data in the Heegaard diagram \cite[Equation 9]{OSzKnot}. When the Heegaard diagram is obtained by gluing together two bordered Heegaard diagrams, this formula splits into a sum of compositions from the two sides; this is precisely the sum in the second line of \eqref{eq:Aglue-proof}. In the setting of cabling, Hom computes the contribution from the $A$ side directly from a Heegaard diagram for the pattern knot and writes the contribution from the $D$ side as an explicit linear function of the Alexander grading in $\CFK(S^3,K)$; the sum of these contributions is precisely \eqref{eq:Aglue}. In our setting, because we are not computing $\CFAm$ directly from a Heegaard diagram, we are forced to determine the constant $C_a$ indirectly, as explained above.

Another strategy for determining absolute Alexander gradings, used by the author in \cite{LevineDoublingOperators} and by Petkova in \cite{PetkovaCables}, is first to compute the relative Alexander grading on $\CFKm(Y,P(K))$ as described in \cite{LOTBordered}, and then to pin down the absolute grading using the symmetry of knot Floer homology \cite[Equation 3]{OSzKnot}, which refines the symmetry of the Alexander polynomial. The disadvantage of this approach is that it requires careful consideration of all the generators of the tensor product complex and use of the non-abelian grading on bordered Floer homology, rather than only the generators that affect $\tau(P(K))$.
\end{remark}

\subsection{\texorpdfstring{$\CFD$}{CFD} of knot complements}

We now recall Lipshitz, Ozsv\'ath, and Thurston's formula for $\CFD(X_K)$ in terms of $\CFKm(S^3,K)$ \cite[Theorems 11.27 and A.11]{LOTBordered}, using some notation from \cite[Section 2.4]{HeddenLevineSplicing}. Let $C^- =\CFKm(S^3,K)$. The following is a slight enhancement of \cite[Proposition 2.5]{HeddenLevineSplicing}:

\begin{proposition} \label{prop:bases}
There exist a pair of bases $\{\tilde \xi_0, \dots, \tilde \xi_{2n}\}$ and $\{\tilde \eta_0, \dots, \tilde \eta_{2n}\}$ for $\CFKm(S^3,K)$ (over $\F[U]$) satisfying:

\begin{enumerate}
\item \label{item:vertsimp}
$\{\tilde\xi_0, \dots, \tilde\xi_{2n}\}$ is a vertically simplified basis, with a vertical arrow of length $k_j \ge 1$ from $\tilde\xi_{2j-1}$ to $\tilde\xi_{2j}$ for each $j = 1, \dots, n$.

\item \label{item:horizsimp}
$\{\tilde\eta_0, \dots, \tilde\eta_{2n}\}$ is a horizontally simplified basis, with a horizontal arrow of length $l_j \ge 1$ from $\tilde\eta_{2j-1}$ to $\tilde\eta_{2j}$ for each $j = 1, \dots, n$.

\item \label{item:epsilon}
If $\epsilon(K) = -1$, then $\tilde\xi_0 = \tilde\eta_1$ and $\tilde\eta_0 = \tilde\xi_1$. If $\epsilon(K) = 0$, then $\tilde\xi_0 = \tilde\eta_0$. If $\epsilon(K) = 1$, then $\tilde\xi_0 = \tilde\eta_2$ and $\tilde\eta_0 = \tilde\xi_2$.

\item \label{item:changeofbasis}
If
\begin{equation} \label{eq:changeofbasis}
\tilde\xi_p = \sum_{q=0}^{2n} \tilde a_{p,q} \tilde\eta_q \quad \text{and} \quad \tilde\eta_p = \sum_{q=0}^{2n} \tilde b_{p,q} \tilde\xi_q,
\end{equation}
where $\tilde a_{p,q}, \tilde b_{p,q} \in \F[U]$, then $\tilde a_{p,q}=0$ whenever $A(\tilde\xi_p) \ne A(\tilde a_{p,q} \tilde\eta_q)$, and $\tilde b_{p,q} = 0$ whenever $A(\tilde\eta_p) \ne A(b_{p,q} \tilde\xi_q)$. (In other words, each $\tilde\xi_p$ is an $\F[U]$-linear combination of the elements $\tilde\eta_q$ that are in the same filtration level as $\tilde\xi_p$, and vice versa.) Define $a_{p,q} = \tilde a_{p,q}|_{U=0}$ and $b_{p,q} = \tilde b_{p,q}|_{U=0}$.

\item \label{item:tau}
$A(\tilde\xi_0) = \tau(K)$ and $A(\tilde\eta_0) = -\tau(K)$.
\end{enumerate}
\end{proposition}

\begin{proof}
By \cite[Proposition 2.4]{HeddenLevineSplicing}, we may find bases $\{\tilde \xi_0, \dots, \tilde \xi_{2n}\}$ and $\{\tilde \eta_0, \dots, \tilde \xi_{2n}\}$ satisfying conditions \ref{item:vertsimp}, \ref{item:horizsimp}, \ref{item:changeofbasis}, and \ref{item:tau}, and such that $\tilde \xi_0$ equals either $\tilde \eta_0$, $\tilde \eta_1$, or $\tilde \eta_2$ as in condition \ref{item:epsilon}.

If $\epsilon(K) =0$, we are done; otherwise, we modify the horizontally simplified basis as follows. Suppose $\epsilon(K)=-1$. By the symmetry of knot Floer homology, the distinguished horizontal generator $\tilde \eta_0$ has an outgoing vertical differential, which implies that $b_{0,2j-1} \ne 0$ for some $j \in \{1,\dots,n\}$. After reordering the elements of $\{\tilde \xi_1, \dots, \tilde \xi_{2n}\}$, we may assume that $b_{0,1} = 1$, and that for any other $j$ with $b_{0,2j-1} = 1$, we have $k_j \ge k_1$ and hence $A(\tilde \xi_{2j}) \le A(\tilde \xi_2)$. Thus, replacing $\tilde \xi_1$ and $\tilde \xi_2$ with $\tilde \xi_1' = \tilde \eta_0$ and $\tilde \xi_2' = \sum_{j=1}^n b_{0,2j-1} \tilde \xi_{2j}$ is a filtered change of basis, and $\partial \tilde \xi_1' \equiv \tilde \xi_2' \pmod {U \cdot C^-}$. The new bases satisfy all the conclusions of the theorem. The case where $\epsilon(K)=1$ is treated similarly.
\end{proof}

\begin{remark}
Property \ref{item:epsilon} may be taken as the definition of $\epsilon(K)$; Hom \cite{HomTau} proves that it does not depend on the choice of bases.
\end{remark}

\begin{theorem} \label{thm:cfkcfd}
Let $K$ be a knot in $S^3$. Given bases $\{\tilde \xi_0, \dots, \tilde \xi_{2n}\}$ and $\{\tilde \eta_0, \dots, \tilde \eta_{2n}\}$ satisfying the conclusions of Proposition \ref{prop:bases}, the type-$D$ structure $\CFD(X_K)$ satisfies the following properties:

\begin{itemize}
\item The summand $\iota_0 \cdot \CFD(X_K)$ has dimension $2n+1$, with designated bases $\{\xi_0, \dots, \xi_{2n}\}$ and $\{\eta_0, \dots, \eta_{2n}\}$ related by
\[
\xi_p = \sum_{q=0}^{2n} a_{p,q} \eta_q \quad \text{and} \quad \eta_p = \sum_{q=0}^{2n} b_{p,q} \xi_q.
\]
These elements are all homogeneous with respect to the grading by relative spin$^c$ structures.

\item The summand $\iota_1 \cdot \CFD(X_K)$ has dimension $\sum_{j=1}^n (k_j + l_j) + s$, where $s = 2 \abs{\tau(K)}$, with basis
\[
\bigcup_{j=1}^n \{\kappa^j_1, \dots, \kappa^j_{k_j}\} \cup \bigcup_{j=1}^n \{\lambda^j_1, \dots, \lambda^j_{l_j}\} \cup \{\mu_1, \dots, \mu_s\} .
\]

\item For $j=1,\dots, n$, corresponding to the vertical arrow $\tilde\xi_{2j-1} \to \tilde\xi_{2j}$, there are coefficient maps
\begin{equation} \label{eq:vertchain}
\xi_{2j} \xrightarrow{D_{123}} \kappa^j_1 \xrightarrow{D_{23}} \cdots \xrightarrow{D_{23}} \kappa^j_{k_j} \xleftarrow{D_1} \xi_{2j-1}.
\end{equation}

\item For $j=1, \dots, n$, corresponding to the horizontal arrow $\tilde\eta_{2j-1} \to \tilde\eta_{2j}$, there are coefficient maps
\begin{equation} \label{eq:horizchain}
\eta_{2j-1} \xrightarrow{D_3} \lambda^j_1 \xrightarrow{D_{23}} \cdots \xrightarrow{D_{23}} \lambda^j_{l_j} \xrightarrow{D_2} \eta_{2j},
\end{equation}

\item Depending on $\tau(K)$, there are additional coefficient maps
\begin{equation} \label{eq:unstchain}
\begin{cases}
\eta_0 \xrightarrow{D_3} \mu_1 \xrightarrow{D_{23}} \cdots \xrightarrow{D_{23}} \mu_s \xleftarrow{D_1} \xi_0 & \tau(K)>0  \\
\xi_0 \xrightarrow{D_{12}} \eta_0 & \tau(K) =0 \\
\xi_0 \xrightarrow{D_{123}} \mu_1 \xrightarrow{D_{23}} \cdots \xrightarrow{D_{23}} \mu_s \xrightarrow{D_2} \eta_0 & \tau(K) < 0.
\end{cases}
\end{equation}
\end{itemize}
\end{theorem}

We refer to the subspaces of $\CFD(X_K)$ spanned by the generators in \eqref{eq:vertchain}, \eqref{eq:horizchain}, and \eqref{eq:unstchain} as the \emph{vertical chains}, \emph{horizontal chains}, and \emph{unstable chain}, respectively.\footnote{Note that our notation differs slightly from that of \cite{LOTBordered}: the generators $\kappa^j_1, \dots, \kappa^j_{k_j}$ are indexed in the reverse order, as are $\mu_1, \dots, \mu_s$ in the case where $\tau(K)>0$.}

We conclude this section with a pair of technical lemmas that will be needed in Section \ref{sec:epsilon}. They are somewhat similar in flavor to results in \cite[Section 3]{HeddenLevineSplicing}.

\begin{lemma} \label{lemma:D1D2D3(xi0)}
In $\CFD(X_K)$ for any knot $K \subset S^3$, we have $(D_1 \circ D_2 \circ D_3)(\xi_0) = 0$.
\end{lemma}

\begin{proof}
The only way we may have $(D_2 \circ D_3)(\xi_0) \ne 0$ is if $\epsilon(K)=-1$ and $l_1=1$, so that there is a horizontal chain
\[
\xi_0 = \eta_1 \xrightarrow{D_3} \lambda^1_1 \xrightarrow{D_2} \eta_2.
\]
We thus must show in this case that $D_1(\eta_2) = 0$.

Since the horizontal arrow from $\tilde\eta_1 = \tilde\xi_0$ to $\tilde\eta_2$ has length $1$, we have $A(\tilde \eta_1) = A(\tilde\xi_0) = \tau(K)$ and $A(\tilde\eta_2) = \tau(K)+1$. By the definition of a horizontally simplified basis, $\partial \tilde\eta_1 = U \tilde\eta_2 + \gamma$, where $A(\gamma)<\tau(K)$; by the definition of a vertically simplified basis, $\partial \tilde\eta_1 \in U \cdot C^-$, so $\gamma = U \delta$, and $A(\delta) \le \tau(K)$. We have
\[
0 = \partial^2 \tilde\eta_1 = U \partial \tilde\eta_2 + \partial \gamma = U \partial( \tilde\eta_2 + \delta),
\]
and since multiplication by $U$ is injective, $\partial \tilde \eta_2 = \partial \delta$.

From \eqref{eq:changeofbasis}, we have
\[
\tilde\eta_2 = \sum_{q=0}^{2n} \tilde b_{2,q} \tilde\xi_q,
\]
where $A(\tilde b_{2,q} \tilde \xi_q) = \tau(K)+1$ whenever $\tilde b_{2,q} \ne 0$. Recall that $b_{2,q} = \tilde b_{2,q}|_{U=0}$. We may also write
\[
\gamma = \sum_{q=0}^{2n} \tilde c_q \tilde \xi_q
\]
for some polynomials $\tilde c_0, \dots, \tilde c_{2n} \in \F[U]$, where $A(\tilde c_q \tilde \xi_q) \le \tau(K)$ whenever $\tilde c_q \ne 0$, and set $c_q = \tilde c_q | _{U=0}$. The conditions on the Alexander grading imply that $c_q$ and $b_{2,q}$ cannot both be nonzero for any $q$. Now, by the definition of a vertically simplified basis,
\[
\sum_{j=1}^n b_{2,2j-1} \tilde\xi_{2j} \equiv \partial \tilde\eta_2 = \partial \delta \equiv \sum_{j=1}^n c_{2j-1} \tilde\xi_{2j} \pmod {U \cdot C^-}.
\]
Therefore, $b_{2,2j-1} = c_{2j-1}=0$ for all $j=1, \dots, n$. Returning to $\CFD(X_K)$, we see that $\eta_2$ is a linear combination of $\{\xi_2, \xi_4, \dots, \xi_{2n}\}$, which completes the proof.
\end{proof}

\begin{lemma} \label{lemma:D3(xi2)}
Suppose $K$ is a knot in $S^3$ such that $\epsilon(K)=-1$ and the vertical arrow from $\tilde \eta_0 = \tilde \xi_1$ to $\tilde \xi_2$ has length $k_1=1$. Then, in $\CFD(X_K)$, we have $D_3(\xi_2)=0$.
\end{lemma}

\begin{proof}
An argument similar to that of the previous lemma shows that $\xi_2$ is a linear combination of $\{\eta_2, \eta_4, \dots, \eta_{2n}\}$.
\end{proof}

\section{Bordered Floer homology of two-bridge link complements} \label{sec:2bridge}

\begin{figure}
\labellist \small
 \pinlabel $L_2$ [r] at 4 161
 \pinlabel $L_1$ [l] at 64 161
 \pinlabel $A$ [t] at 35 5
\endlabellist
\includegraphics[scale=0.8]{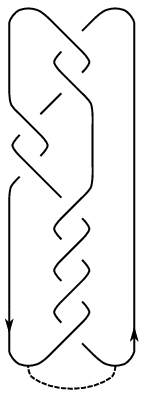}
\caption{The two-bridge link $L_Q$, along with an arc joining the two components.}
\label{fig:2bridge}
\end{figure}

Let $L = L_1 \cup L_2$ be any two-component, two-bridge link, presented by a plat diagram such as the one shown in Figure \ref{fig:2bridge}. (We do not require the diagram to be alternating, although such diagrams can always be found.) Orient $L$ such that both $L_1$ and $L_2$ are oriented counterclockwise in the projection plane near their local minima. The components of $L$ are both unknotted, meaning that the complement of either component is a solid torus. (For the specific link $L_Q$ shown in Figure \ref{fig:2bridge}, the remaining component, viewed as a knot in the solid torus, is precisely the Mazur pattern knot $Q$ shown in Figure \ref{fig:Q}.) Let $X(L)$ be the strongly bordered manifold $S^3 \minus \nbd(L)$, where each of the boundary components is equipped with the $0$-framing; we connect the two boundary components of $X(L)$ using an arc $A$ in the projection plane connecting the two local minima, equipped with the blackboard framing. The goal of this section is to describe how to compute the bordered invariant $\CFDD(X(L))$ using the arc slides algorithm of Lipshitz, Ozsv\'ath, and Thurston \cite{LOTFactoring}, and then to provide the computation for $\CFDD(X(L_Q))$.

\begin{figure}
\labellist \small
 \pinlabel* {$a_1$} [b] at 156 -5
 \pinlabel* {$b_1$} [b] at 140 -5
 \pinlabel* {$a_2$} [b] at 124 -5
 \pinlabel* {$b_2$} [b] at 108 -5
 \pinlabel* {$c_1$} [b] at 92 -5
 \pinlabel* {$d_1$} [b] at 76 -5
 \pinlabel* {$c_2$} [b] at 60 -5
 \pinlabel* {$d_2$} [b] at 44 -5
 \pinlabel* {{\color{green} $\alpha_1$}} at 175 50
 \pinlabel* {{\color{darkgreen} $\alpha_2$}} at 144 50
 \pinlabel* {{\color{red} $\alpha_3$}} at 79 50
 \pinlabel* {{\color{darkred} $\alpha_4$}} at 46 50
\endlabellist
\includegraphics[scale=0.8]{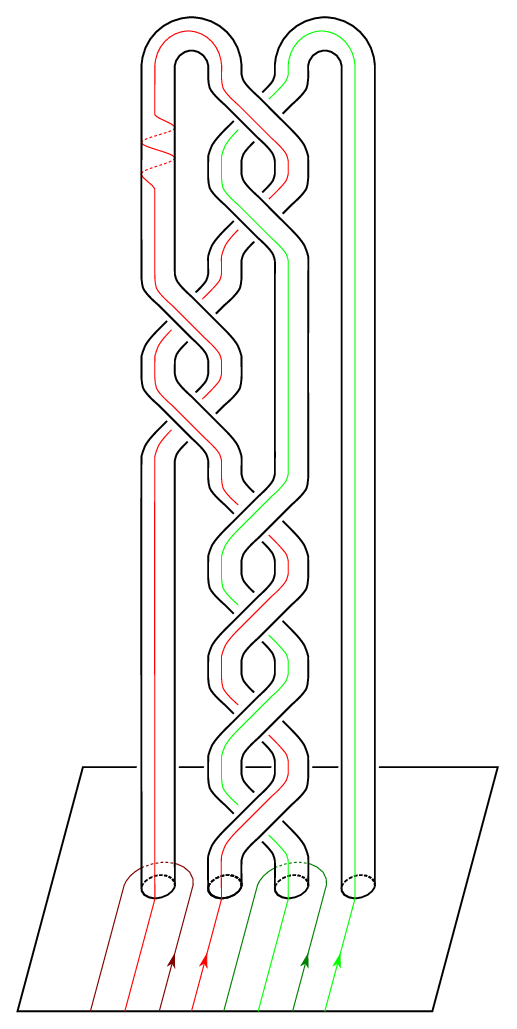}
\caption{A bordered Heegaard diagram for the complement of the two-bridge link in Figure \ref{fig:2bridge}. The positive $x$ axis points into the page, the positive $y$ axis points to the left, and the positive $z$ axis points upward; with this convention, the $(x,y)$ plane can be identified with the diagram in Figure \ref{fig:H0}.}
\label{fig:2bridge-heegaard}
\end{figure}
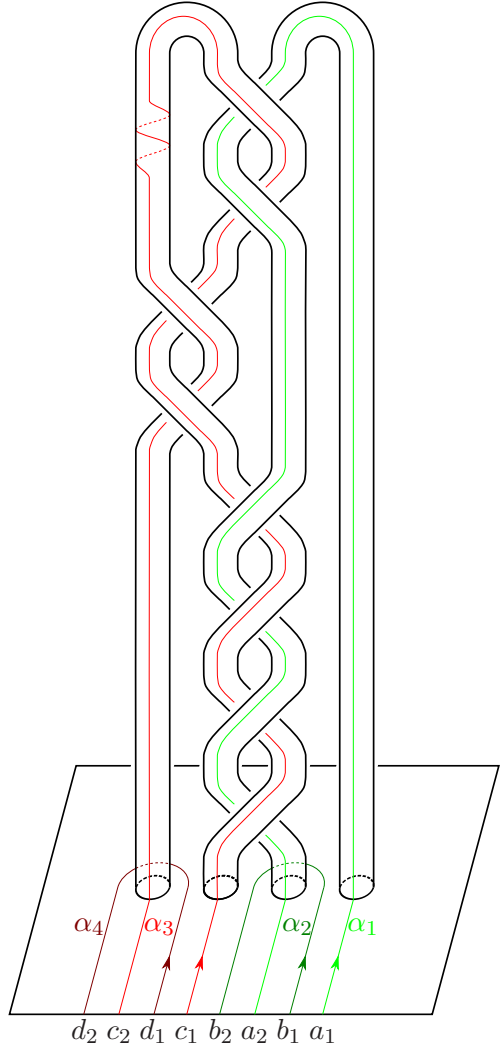

Let $X_{\dr}(L) = X(L) \minus \nbd(A)$, and notice that $X_{\dr}$ is in fact a genus-2 handlebody. To describe the parametrization of $\partial X_{\dr}(L)$ in terms of arc slides, it helps to consider how to obtain a bordered Heegaard diagram $\HH_{\dr}(L)$ for $X_{\dr}(L)$. Let $(x,y,z)$ denote coordinates on $\R^3$, where $z$ is the height function with respect to which $L$ is in bridge position. We may view a neighborhood of $L_1 \cup L_2 \cup A$ in $S^3$ as the lower half-space $\{z < 0\}$ in $\R^3$ together with a pair of interlocking $1$-handles in the upper half-space whose feet lie along the $y$-axis. $X_{\dr}(L)$ is then the complement of this configuration, consisting of the upper half-space $\{z \ge 0\}$ minus the two one-handles, together with the point at $\infty$. Since $X_{\dr}$ is a handlebody, its boundary (which consists of the $xy$-plane together with the point at infinity, minus four disks, plus two tubes in the upper half-space) is a Heegaard surface for $X_{\dr}(L)$. The distinguished disk $\Delta$ may be taken to be a neighborhood of the point at infinity, with the basepoint $z$ lying on the $y$ axis; let $\bar \Sigma = \partial X_{\dr} \minus \Delta$, with orientation coming from the standard orientation of $\R^2$ (which is opposite the boundary orientation on $\partial X_{\dr}(L)$). The $\alpha$ arcs $\alpha_1, \alpha_2, \alpha_3, \alpha_4$ are chosen to satisfy:
\begin{itemize}
\item
$\alpha_1$ (resp.~$\alpha_3$) consists of a pair of arcs in the $\{x\le 0, z = 0\}$ half-plane connecting $\partial \bar\Sigma$ to the feet of the $1$-handle corresponding to $L_1$ (resp.~$L_2$), together with a $0$-framed longitudinal arc in the boundary of the $1$-handle.
\item
$\alpha_2$ (resp.~$\alpha_4$) consists of a pair of arcs in the $\{x\le 0, z = 0\}$ half-plane connecting $\partial \bar\Sigma$ to the $y$-axis, joined by an arc in the $\{x \ge 0, z=0\}$ half-plane.
\item
If we label the endpoints of $\alpha_1$, $\alpha_2$, $\alpha_3$, $\alpha_4$ by $(a_1, a_2)$, $(b_1,b_2)$, $(c_1, c_2)$, and $(d_1,d_2)$, respectively, and traverse $\partial \bar \Sigma$ with the orientation opposite to that induced from $\bar \Sigma$, we encounter the points in the order $a_1, b_1, a_2, b_2, c_1, d_1, c_2, d_2$.
\end{itemize}
To find the $\beta$ circles, we apply an ambient isotopy of $\R^3$ that untangles the two $1$-handles from each other so that they become separated by the $(x,z)$ plane, at which point the compression disks become evident. We must describe the effect of the isotopy on the $\alpha$ arcs.

For simplicity, assume that the bridge presentation of $L$ consists entirely of whole twists and that the top and bottom closures both consist of arcs that are neither overlapping nor nested. To be precise, suppose that $L$ is the plat closure of the $4$-stranded braid
\begin{equation} \label{eq:braidword}
\sigma_2^{2a_0} \sigma_1^{2a_1} \cdots \sigma_2^{2a_{m-2}} \sigma_1^{2a_{m-1}} \sigma_2^{2a_{m}}
\end{equation}
for some nonzero integers $a_0, \dots, a_m$, where $m$ is even and the braid is read from bottom to top. With the orientation on $L$ given above, the linking number of $L_1$ and $L_2$ is
\[
\ell = \lk(L_1, L_2)  = - \sum_{i=0}^{m/2} a_{2i}.
\]
(For our example $L_Q$, $m=2$ and $(a_0, a_1, a_2) = (2,-1,-1)$, so $\ell= -1$.)

\begin{figure}
\labellist \small
 \pinlabel* {$a_1$} [r] at 8 23
 \pinlabel* {$b_1$} [r] at 8 37
 \pinlabel* {$a_2$} [r] at 8 51
 \pinlabel* {$b_2$} [r] at 8 65
 \pinlabel* {$c_1$} [r] at 8 87
 \pinlabel* {$d_1$} [r] at 8 101
 \pinlabel* {$c_2$} [r] at 8 115
 \pinlabel* {$d_2$} [r] at 8 129
 \pinlabel* {{\color{green} $\alpha_1$}} [b] at 18 23
 \pinlabel* {{\color{darkgreen} $\alpha_2$}} [b] at 18 37
 \pinlabel* {{\color{red} $\alpha_3$}} [b] at 18 115
 \pinlabel* {{\color{darkred} $\alpha_4$}} [b] at 18 129
 \pinlabel* {$\gamma_2$} [l] at 50 69
 \pinlabel* {$\gamma_1$} [l] at 50 101
\endlabellist
\includegraphics{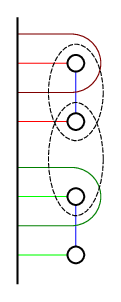}
\caption{The standard Heegaard diagram for the $0$-framed genus-$2$ handlebody.}
\label{fig:H0}
\end{figure}
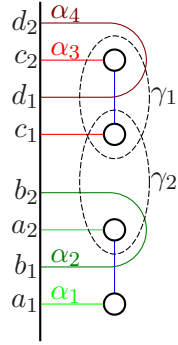

Following the notation of \cite[Section 1.4]{LOTFactoring}, let $\ZZ^g$ denote the split pointed matched circle of genus $g$. Let $H^g$ denote a handlebody of genus $g$, and let $\phi_g^0\co -F(\ZZ) \to \partial H^g$ denote the standard ($0$-framed) parametrization of the boundary; in the case $g=2$, $(H^2, \phi_2^0)$ is represented by the Heegaard diagram $\HH_0$ in Figure \ref{fig:H0} (which is diffeomorphic to the one in \cite[Figure 5]{LOTFactoring}.) Let $\gamma_1, \gamma_2$ be curves in $\HH_0$ as shown. Let $T_1$ denote a positive Dehn twist around $\gamma_1$, and let $T_2$ denote a positive Dehn twist around $\gamma_2$ composed with a negative Dehn twist around the meridian of each of the two tubes. If we ignore the $\beta$ circles, we may identify $\HH_0$ with the surface in Figure \ref{fig:2bridge-heegaard}. Undoing a full right-handed twist between the feet of the left-hand $1$-handle in Figure \ref{fig:2bridge-heegaard} modifies the curves in the $(x,y)$ plane by $T_1$, and undoing a full twist between the feet of the two $1$-handles modifies the curves by $T_2$. Applying these operations and their inverses in the sequence prescribed by \eqref{eq:braidword}, we see:

\begin{lemma} \label{lemma:psi_L}
The bordered Heegaard diagram $\HH_{\dr}(L)$ is isotopic to the diagram obtained from $\HH_0$ by applying the diffeomorphism
\begin{equation} \label{eq:psi_L}
\psi_L = T_2^{a_m} \circ T_1^{a_{m-1}} \circ T_2^{a_{m-2}} \circ \cdots \circ T_1^{a_1} \circ T_2^{a_0}
\end{equation}
to the $\alpha$ arcs and leaving the $\beta$ circles unchanged. Therefore, as a bordered manifold,
\[
X_{\dr}(L) = (H^2, \phi_2^0 \circ \psi_L).
\]
\end{lemma}

It remains to describe the factorization of $T_1$ and $T_2$ into arc-slides. For this discussion, we label the points on the boundary of any genus-$2$ pointed matched circle $p_0, \dots, p_7$. An arc-slide of $p_i$ over $p_j$ (where $j = i \pm 1$) is indicated by $[i \to j]$, assuming that the initial and final pointed matched circles are known from context. Composition is written from right to left (just as with functions).


\begin{proposition} \label{prop:arcslides}
In the strongly based mapping class groupoid of genus 2, the mapping classes
\[
T_1, T_2 \co F(\ZZ^2) \to F(\ZZ^2)
\]
have the following factorizations into arc-slides:
\begin{align*}
T_1 &= [7 \to 6]^6 \\
T_2 &= [6 \to 5] \circ [5 \to 4] \circ [4 \to 3] \circ [7 \to 6] \circ [5 \to 4] \circ [4 \to 3] \circ [3 \to 2] \circ {} \\
& \qquad [6 \to 5] \circ [4 \to 3] \circ [3 \to 2] \circ [2 \to 1] \circ [3 \to 4] \circ [4 \to 5] \circ [4 \to 3]
\end{align*}
\end{proposition}

\begin{figure}
\includegraphics{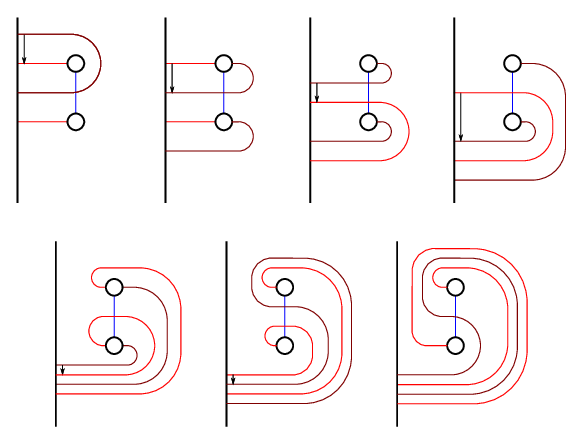}
\caption{Arc slide decomposition of $T_1$.} \label{fig:arcslides-T1}
\end{figure}

\begin{figure}
\includegraphics{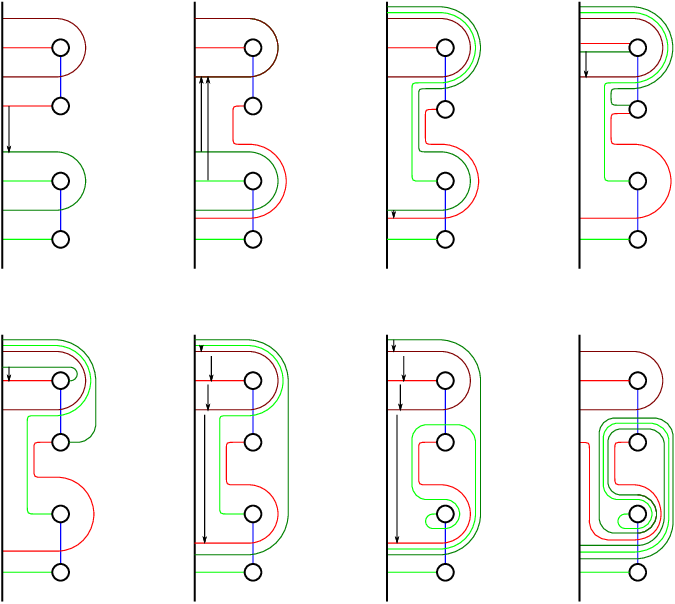}
\caption{Arc slide decomposition of $T_2$.} \label{fig:arcslides-T2}
\end{figure}

\begin{proof}
Let $T_i(\HH_0)$ denote the bordered Heegaard diagram obtained by applying $T_i$ to the $\alpha$ arcs of $\HH_0$ while leaving the $\beta$ circles unchanged. An element of the mapping class groupoid is completely determined by where it sends the $\alpha$ arcs of $\HH_0$. Therefore, if we exhibit a sequence of arc-slides taking the $\alpha$ arcs of $\HH_0$ to those of $T_i(\HH_0)$, it follows that the sequence is a factorization of $T_i$.

The verification is shown in Figures \ref{fig:arcslides-T1} and \ref{fig:arcslides-T2}. For $T_1$, all of the arc-slides take place within the connected summand of the surface containing the arcs $\alpha_3$ and $\alpha_4$, so we only show that summand. Figure \ref{fig:arcslides-T1} shows that performing six arc-slides of the topmost point over the arc adjacent to it ($[7 \to 6]$ in the above notation) results in a Dehn twist along a curve encircling both feet of the $1$-handle.

For $T_2$, the sequence of arc-slides is more complicated. We first slide $\alpha_3$ over $\alpha_2$ ($[4 \to 3]$), and then slide $\alpha_2$ and $\alpha_3$ over $\alpha_4$ ($[4 \to 5]$ followed by $[3 \to 4]$). We then slide $\alpha_2$ over $\alpha_3, \alpha_4, \alpha_3$ so that the feet of $\alpha_3$ and $\alpha_4$ become nested between those of both $\alpha_1$ and $\alpha_2$. We then perform a sequence of four slides of $\alpha_1$ over $\alpha_3$ and $\alpha_4$ alternately, and then do the same for $\alpha_4$. It is easy to verify that the resulting diagram agrees the one obtained by applying $T_2$ to $\HH_0$ and that the notation for these arc-slides agrees with the composition in the statement of the proposition.
\end{proof}

\begin{remark}
In the more general setting where the $2$-bridge diagram for $L$ is not comprised of full twists only, one can compute the arc-slide factorizations of the surface diffeomorphisms corresponding to undoing a single crossing rather than a full twist by the same techniques. The initial and final pointed matched circles of each of these diffeomorphisms are not the same, so there are more cases to consider. Indeed, a similar strategy can be used to compute the bordered invariants of any knot or link complement, starting from a bridge presentation.
\end{remark}

According to \cite[Theorem 3]{LOTFactoring}, we may compute $\CFD(X_{\dr}(L))$ by factoring $\psi_L$ into arc-slides, which can be done combining Lemma \ref{lemma:psi_L} and Proposition \ref{prop:arcslides}. Specifically, suppose that $\psi_L = \psi_1 \circ \cdots \circ \psi_n$, where $\psi_i \co F(\ZZ_i) \to F(\ZZ_{i-1})$ is an arc-slide, where $\ZZ_0 = \ZZ_n = \ZZ^2$. By \cite[Theorem 3]{LOTFactoring},
\begin{multline} \label{eq:CFD(Xdr)}
\CFD(X_{\dr}(L)) \simeq \Mor \left( \DDhat(\I_{\ZZ_{n-1}}) \otimes \cdots \otimes \DDhat(\I_{\ZZ_{1}}), \right. \\
\left. \DDhat(\psi_n) \otimes \cdots \otimes \DDhat (\psi_1) \otimes \CFD(H^2, \phi_2^0) \right).
\end{multline}
Here $\DDhat(\I_{\ZZ_i})$ denotes the identity $DD$ bimodule for the algebra $\AA(\ZZ_i)$, $\DDhat(\psi_i)$ denotes the arc-slide $DD$ bimodule associated to $\psi_i$, all tensor products are taken over the appropriate rings of idempotents, and $\Mor$ denotes the chain complex of morphisms of $\AA(\ZZ_{n-1}) \otimes \cdots \otimes \AA(\ZZ_1)$--bimodules. All the modules in \eqref{eq:CFD(Xdr)} are completely computed in \cite{LOTFactoring}, allowing for algorithmic computation of $\CFD(X_{\dr}(L))$.

To obtain an arced Heegaard diagram $\HH(L)$ for $X(L)$ \cite[Definition 5.4]{LOTBimodules}, we attach a $2$-dimensional $1$-handle to $\partial \bar\Sigma$, joining $z$ to the segment between $b_2$ and $c_1$, and let the cocore of this $1$-handle be the basepoint arc. Denote the new Heegaard surface $\tilde \Sigma$, and denote its two boundary components $\partial_1 \tilde \Sigma$ and $\partial_2 \tilde \Sigma$, corresponding to the complements of $L_1$ and $L_2$. The $DD$ bimodule $\CFDD(X(L))$ is by definition induced from $\CFD(X_{\dr}(L))$ via the canonical inclusion map $\AA(\ZZ^1) \otimes \AA(\ZZ^1) \to \AA(\ZZ^2)$ \cite[Definition 6.4]{LOTBimodules}. To be precise, denote the copies of $\AA(\ZZ^1)$ corresponding to $\partial_1 \tilde \Sigma$ (resp.~$\partial_2 \tilde \Sigma$) by $\AA_\rho$ (resp.~$\AA_\sigma$), with idempotents $\iota_0^\rho, \iota_1^\rho$ (resp.~$\iota_0^\sigma, \iota_1^\sigma$) and Reeb elements $\rho_I$ (resp.~$\sigma_I$) for contiguous subsequence $I \subset (1,2,3)$. That is, the algebra elements associated to the Reeb chords in $\HH(L)$ are as follows:
\[
\begin{array}{|c|c|c|c|c|c|}
  \hline
  [a_1,b_1] & [b_1,a_2] & [a_2,b_2] & [c_1,d_1] & [d_1,c_2] & [c_2,d_2] \\ \hline
  \rho_1 & \rho_2 & \rho_3 & \sigma_1 & \sigma_2 & \sigma_3 \\ \hline
\end{array}
\]

We also consider the periodic domains in $\HH_{\dr}(L)$ and $\HH(L)$. Orient the $\alpha$ arcs as shown in Figure \ref{fig:2bridge-heegaard}; that is, $\alpha_1$ and $\alpha_3$ run parallel to $L_1$ and $L_2$ with the orientations specified above, and $\alpha_2$ and $\alpha_4$ are \emph{left-handed} meridians of $L_1$ and $L_2$. Letting $[\alpha_i]$ denote the class in $H_1(X_{\dr}(L))$ represented by the union of $\alpha_i$ and a segment of $\partial \bar \Sigma$, we see that $[\alpha_1] = -\ell [\alpha_4]$ and $[\alpha_3] = -\ell [\alpha_2]$. Therefore, in $\HH_{\dr}(L)$, there are periodic domains $\PP_1$ and $\PP_2$ (corresponding to punctured Seifert surfaces for $L_1$ and $L_2$ respectively) whose multiplicities in the eight boundary regions (beginning with the region containing $z$ and ordered opposite to the boundary orientation on $\partial \bar\Sigma$) are given by:
\begin{align*}
\partial^\partial(\PP_1) &= (0,1,1,0,0,0,\ell,\ell) \\
\partial^\partial(\PP_2) &= (0,0,\ell,\ell,0,1,1,0).
\end{align*}
These can also be viewed as periodic domains in $\HH(L)$.

Now suppose $P \subset V$ is a pattern knot with winding number $m$, as in Section \ref{sec:satellite}. Let $(\HH',z,w)$ be a doubly pointed Heegaard diagram for $(V,P)$, and label the boundary regions $R_0,R_1,R_2,R_3$ according to the convention for $\CFA$. There is a periodic domain $\PP_\mu$ in $\HH'$ with boundary multiplicities $(0,0,1,1)$, and $n_w(\PP_\mu) = m$. If we glue the diagrams $\HH'$ and $\HH(L)$ along $\partial_1 \tilde \Sigma$, we obtain a Heegaard diagram for the complement of $L_2$, which is identified with $V$, and the basepoint $w$ determines the pattern knot $P(L_1) \subset V$. The group of periodic domains for the new Heegaard diagram is generated by $\PP_2 + \ell \PP_V$. The multiplicity of this periodic domain at $w$ is $\ell m$, which is thus the winding number of $P(L_1)$.

\subsection{Computation of \texorpdfstring{$\CFDD(X(L_Q))$}{CFDD(X(L\_Q))}}

Bohua Zhan has written a software package in Python that implements the arc-slides algorithm for computing (bordered) Heegaard Floer homology \cite{ZhanComputations}. Specifically, this package contains functions for manipulating (bi)modules over the bordered algebras, including evaluating tensor products and Mor pairings, simplifying modules using the edge reduction algorithm, and recovering $\CFDD(X)$ from $\CFDD(X_{dr})$. It can also generate the type $D$ structure associated to the standard handlebody $(H^g, \phi_g^0)$ of any genus and the type $DD$ bimodule associated to any arc-slide, using the descriptions given in \cite{LOTFactoring}. With Zhan's assistance, the author used this program to compute $\CFDD(X(L_Q))$; the result is given by the following theorem:

\begin{theorem}
\allowdisplaybreaks
Let $L_Q$ denote the two-bridge link shown in Figure \ref{fig:2bridge}. The type $DD$ bimodule $\CFDD(X(L_Q))$ has a basis $\{g_1, \dots, g_{34}\}$ with the following properties:\footnote{The indexing of the generators is completely determined by the output of Zhan's program.}
\begin{enumerate}
\item The associated idempotents in $\AA_\rho$ and $\AA_\sigma$ of the generators are:
\[
\begin{array}{|c|c|c|} \hline
 & \iota_0^\rho & \iota_1^\rho \\ \hline
 \iota_0^\sigma & g_2, g_{21}, g_{27}, g_{29}, g_{34} &
    g_4, g_5, g_9, g_{10}, g_{16}, g_{23}, g_{25}, g_{32} \\ \hline
 \iota_1^\sigma & g_1, g_7, g_{11}, g_{13}, g_{15}, g_{18}, g_{19}, g_{30} &
    g_3, g_6, g_8, g_{12}, g_{14}, g_{17}, g_{20}, g_{22}, g_{24}, g_{26}, g_{28}, g_{31}, g_{33} \\ \hline
\end{array}
\]
\item
The differential is as follows:
\begin{align*}
d(g_{1}) &= \rho_1 \cdot g_{24} \\
d(g_{2}) &= \rho_3 \sigma_3 \cdot g_{6} + \rho_{123} \sigma_{123} \cdot g_{8} + \rho_1 \sigma_{123} \cdot g_{12} + \rho_{123} \sigma_1 \cdot g_{17} \\
d(g_{3}) &= \rho_2 \cdot g_{1} \\
d(g_{4}) &= \rho_2 \cdot g_{21} + \sigma_3 \cdot g_{26} \\
d(g_{5}) &= \sigma_1 \cdot g_{31} \\
d(g_{6}) &= \sigma_2 \cdot g_{25} + \rho_2 \cdot g_{30} \\
d(g_{7}) &= \rho_3 \cdot g_{3} + \rho_1 \cdot g_{12} + \rho_{123} \cdot g_{24} + \rho_1 \sigma_{23} \cdot g_{31}  \\
d(g_{8}) &= 0 \\
d(g_{9}) &= \rho_{23} \cdot g_{25} + \rho_2 \sigma_1 \cdot g_{30} \\
d(g_{10}) &= \rho_{23} \sigma_1 \cdot g_{8} + \sigma_1 \cdot g_{12} + \sigma_{123} \cdot g_{31} + \sigma_3 \cdot g_{33}  \\
d(g_{11}) &= \rho_1 \cdot g_{17} + \rho_1 \sigma_{23} \cdot g_{24} + \rho_3 \cdot g_{28} \\
d(g_{12}) &= 0 \\
d(g_{13}) &= \rho_3 \cdot g_{20} + \sigma_2 \cdot g_{27} \\
d(g_{14}) &= 0 \\
d(g_{15}) &= \rho_1 \sigma_2 \cdot g_{25} + \sigma_{23} \cdot g_{30} \\
d(g_{16}) &= \sigma_{123} \cdot g_{8} + \sigma_1 \cdot g_{17} + \sigma_3 \cdot g_{22} \\
d(g_{17}) &= 0 \\
d(g_{18}) &= \sigma_2 \cdot g_{2} + \rho_{123} \sigma_2 \cdot g_{25} + \rho_3 \cdot g_{26}  \\
d(g_{19}) &= \rho_1 \cdot g_{14} \\
d(g_{20}) &= \rho_{23} \cdot g_{6} + \sigma_2 \cdot g_{9} \\
d(g_{21}) &= \rho_{123} \sigma_{123} \cdot g_{14} + \sigma_3 \cdot g_{15} + \rho_1 \sigma_{123} \cdot g_{17} \\
d(g_{22}) &= \sigma_2 \cdot g_{23} \\
d(g_{23}) &= \sigma_1 \cdot g_{8} \\
d(g_{24}) &= 0 \\
d(g_{25}) &= (\rho_{23} \sigma_3 + \sigma_{123}) \cdot g_{8} + \rho_{23} \sigma_1 \cdot g_{14} + (\rho_{23} \sigma_3 + \sigma_{123}) \cdot g_{24} \\
d(g_{26}) &= \sigma_{23} \cdot g_{6} + \rho_2 \cdot g_{15} \\
d(g_{27}) &= \rho_3 \cdot g_{9} + \rho_{123} \sigma_1 \cdot g_{12}  \\
d(g_{28}) &= \rho_2 \cdot g_{19} \\
d(g_{29}) &= \rho_3 \cdot g_{4} + \rho_{123} \sigma_{123} \cdot g_{17} + \sigma_3 \cdot g_{18} \\
d(g_{30}) &= (\rho_{123} + \rho_3 \sigma_{23}) \cdot g_{8} + \rho_1 \sigma_{23} \cdot g_{31} + (\rho_{123} + \rho_3 \sigma_{23}) \cdot g_{24} \\
d(g_{31}) &= 0 \\
d(g_{32}) &= \rho_2 \cdot g_{2} + \sigma_3 \cdot g_{20} \\
d(g_{33}) &= \sigma_2 \cdot g_{5} \\
d(g_{34}) &= \rho_{123} \sigma_{123} \cdot g_{12} + \sigma_3 \cdot g_{13} + \rho_3 \cdot g_{32}
\end{align*}
\end{enumerate}
\end{theorem}

\section{Computation of \texorpdfstring{$\tau(Q(K))$}{\texttau(Q(K))}} \label{sec:tau}

In this section, we prove the first half of Theorem \ref{thm:Q}: For any knot $K \subset S^3$,
\begin{equation} \label{eq:tau(Q(K))}
\tau(Q(K)) =
\begin{cases}
\tau(K) & \text{if } \tau(K) \le 0 \text{ and } \epsilon(K) \in \{0,1\} \\
\tau(K)+1 & \text{if } \tau(K) > 0 \text{ or } \epsilon(K) = -1.
\end{cases}
\end{equation}

\begin{remark} \label{rmk:CFHH}
Some partial results in the direction of \eqref{eq:tau(Q(K))} follow from much simpler considerations. First, note that $Q(K)$ can be turned into $K$ by changing a single positive crossing into a negative crossing. This operation either preserves $\tau$ or decreases $\tau$ by $1$ \cite[Corollary 1.5]{OSz4Genus}, so $\tau(Q(K))$ must equal $\tau(K)$ or $\tau(K)+1$. Both cases were previously known to occur: If $O$ is the unknot, then $Q(O)=O$, so $\tau(Q(O)) = \tau(O) = 0$. On the other hand, Cochran, Franklin, Hedden, and Horn \cite{CochranFranklinHeddenHorn} showed that if a nontrivial knot $K$ admits a Legendrian representative whose Thurston--Bennequin number satisfies $\operatorname{tb}(K) = 2g(K)-1$, then $\tau(Q(K))= \tau(Q(K))+1$; in fact, their proof carries through almost verbatim under the weaker hypothesis that $\operatorname{tb}(K) = 2\tau(K)-1 > 0$. Subsequently, Ray \cite{RayIterates} observed that under the same hypotheses, the iterated satellites $Q^n(K)$ satisfy $\tau(Q^n(K))=\tau(K)+n$ and hence are distinct in concordance. Both of these results follow directly from \eqref{eq:tau(Q(K))}.
\end{remark}

We begin by using the results of the previous section to compute $\CFAm(V,Q)$. Let $C \subset V$ be the knot $S^1 \times \{\pt\}$, specified by the Heegaard diagram in Figure \ref{fig:cables}(a), and let $X = X(L_Q) = V \minus \nbd(Q)$ be the (bordered) exterior of the $2$-bridge link $L_Q$ depicted in Figure \ref{fig:2bridge}. By a suitable version of the pairing theorem, there is a graded homotopy equivalence
\[
\CFDm(V,Q) \simeq \CFAm(V,C) \boxtimes_{\AA_\rho} \CFDD(X).
\]
As seen in Section \ref{sec:2bridge}, this gluing describes the orientation of $Q$ whose winding number is $-1$. Since knot Floer homology is invariant under orientation reversal, this convention does not affect the computation of $\tau(Q(K))$, but we will need the winding number (with sign) in order to apply Proposition \ref{prop:Aglue}.

The invariant $\CFAm(V,C)$ was computed by Lipshitz, Ozsv\'ath, and Thurston \cite[Lemma 11.22]{LOTBordered}:

\begin{lemma} \label{lemma:CFAm(V,C)}
The type $A$ module for the solid torus equipped with its core circle, $\CFAm(V,C)$, has a single generator $a$, with $\AA_\infty$ multiplications given by
\begin{equation} \label{eq:CFAm(V,C)}
m_{3+i}(a, \rho_3, \underbrace{\rho_{23}, \dots, \rho_{23}}_{i}, \rho_2) = U^{i+1} \cdot a
\end{equation}
for all $i \ge 0$.
\end{lemma}

\begin{proposition} \label{prop:CFDm(V,Q)}
The type $D$ structure $\CFDm(V,Q)$ has a basis
\[
\{x_0, \dots, x_6, y_1, \dots, y_6\}
\]
with the following properties:

\begin{itemize}
\item
The generators $x_0, x_2, x_4, y_2, y_4$ are in $\iota_0 \cdot \CFDm(V,Q)$, and the remaining ones are in $\iota_1 \cdot \CFDm(V,Q)$.

\item
The differential is as follows:
\begin{equation} \label{eq:CFDm(V,Q)}
\xymatrix{
x_0 \ar[dr]_{U \sigma_1} & x_1 \ar[l]_{\sigma_2} \ar[d]^{U^2} & x_2 \ar[d]_{U} \ar[l]_{\sigma_3} & x_3 \ar[dl]^{\sigma_2} \ar[d]^{U} & x_4 \ar[l]_{\sigma_3} \ar[d]^{U} & x_5 \ar[d]^{U} & x_6 \ar[d]^{U} \\
& y_1 & y_2 \ar[l]_{U \sigma_3} & y_3 \ar@/^1.5pc/[ll]^{\sigma_{23}} & y_4 \ar[l]^{\sigma_3} & y_5 & y_6
}
\end{equation}
\end{itemize}
\end{proposition}

\begin{proof}
In the tensor product $\CFAm(V,C) \boxtimes_{\AA_\rho} \CFDD(X)$, we define:
\[
\begin{aligned}
x_0 &= a \otimes g_{27} &
x_1 &= a \otimes g_{13} & y_1 &= a \otimes g_{30} &
x_2 &= a \otimes g_{34} & y_2 &= a \otimes g_{2} \\
&& x_3 &= a \otimes g_{18} & y_3 &= a \otimes g_{15} &
x_4 &= a \otimes g_{29} & y_4 &= a \otimes g_{21} \\
&& x_5 &= a \otimes g_{7} & y_5 &= a \otimes g_{11} &
x_6 &= a \otimes g_{1} & y_6 &= a \otimes g_{19}
\end{aligned}
\]

Any differential in $\CFDD(X)$ that involves only $\sigma_I$ results in a differential in the tensor product. The chains of differentials in $\CFDD(X)$ that pair with the higher multiplications in $\CFAm(V,C)$ to produce differentials in the tensor product are:
\begin{align*}
& \xymatrix{ g_2 \ar[r]^{\rho_3 \sigma_3} & g_6 \ar[r]^{\rho_2} & g_{30} } &
& \xymatrix{ g_7 \ar[r]^{\rho_3} & g_3 \ar[r]^{\rho_2} & g_1 } \\
& \xymatrix{ g_{11} \ar[r]^{\rho_3} & g_{28} \ar[r]^{\rho_2} & g_{19} } &
& \xymatrix{ g_{13} \ar[r]^{\rho_3} & g_{20} \ar[r]^{\rho_{23}} & g_6 \ar[r]^{\rho_2} & g_{30} } \\
& \xymatrix{ g_{18} \ar[r]^{\rho_3} & g_{26} \ar[r]^{\rho_2} & g_{15} } &
& \xymatrix{ g_{27} \ar[r]^{\rho_3} & g_9 \ar[r]^{\rho_2 \sigma_1} & g_{30} } \\
& \xymatrix{ g_{29} \ar[r]^{\rho_3} & g_{4} \ar[r]^{\rho_2} & g_{21} } &
& \xymatrix{ g_{34} \ar[r]^{\rho_3} & g_{32} \ar[r]^{\rho_2} & g_{2} }
\end{align*}
These account for the eight arrows in \eqref{eq:CFDm(V,Q)} that involve positive powers of $U$.
\end{proof}

\begin{figure}
\labellist \small
 \pinlabel $z$ at 28 142
 \pinlabel $w$ at 64 15
 \pinlabel $1$ at 151 151
 \pinlabel $2$ at 151 19
 \pinlabel $3$ at 19 19
 \pinlabel* $x_0$ [r] at 8 54
 \pinlabel* $y_2$ [r] at 8 70
 \pinlabel* $y_4$ [r] at 8 86
 \pinlabel* $x_4$ [r] at 8 102
 \pinlabel* $x_2$ [r] at 8 118
 \pinlabel* $x_5$ [t] at 46 5
 \pinlabel* $x_6$ [t] at 57 5
 \pinlabel* $y_6$ [t] at 68 5
 \pinlabel* $y_5$ [t] at 79 5
 \pinlabel* $y_1$ [t] at 90 5
 \pinlabel* $y_3$ [t] at 101 5
 \pinlabel* $x_3$ [t] at 112 5
 \pinlabel* $x_1$ [t] at 123 5
\endlabellist
\includegraphics[scale=1.1]{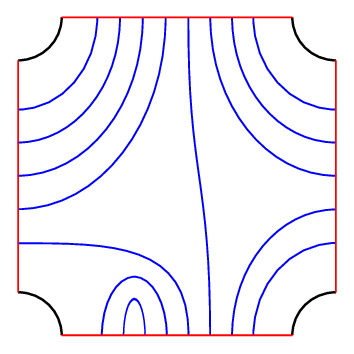}
\caption{A genus-one Heegaard diagram for $(V,Q)$.} \label{fig:Q-heegaard}
\end{figure}
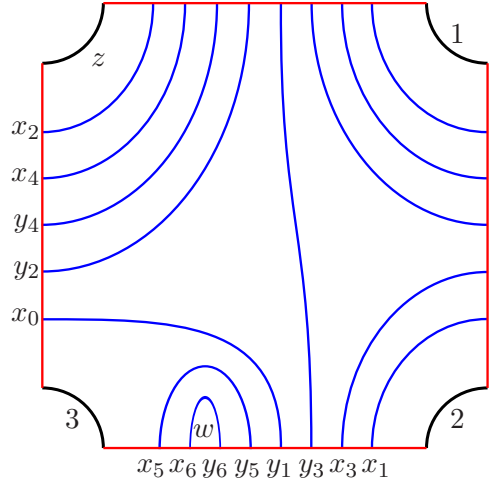

\begin{remark} \label{rmk:Q-heegaard}
The reader may easily verify that the doubly pointed, bordered Heegaard diagram $(\HH,z,w)$ shown in Figure \ref{fig:Q-heegaard} presents $(V,Q)$. If we label the generators of $\CFDm(\HH,z,w)$ as shown, one may use the Riemann mapping theorem to verify that $\CFDm(\HH,z,w)$ agrees with \eqref{eq:CFDm(V,Q)}, except for a few additional differentials:
\begin{align*}
x_2 &\xrightarrow{U \sigma_{123}} x_5 &
 y_2 &\xrightarrow{U \sigma_{123}} y_5 \\
x_4 &\xrightarrow{U \sigma_{123}} x_6 &
 y_4 &\xrightarrow{U \sigma_{123}} y_6.
\end{align*}
However, we may make a change of basis, replacing $y_2$ and $y_4$ with $y_2' = y_2 + \sigma_{123} x_5$ and $y_4' = y_4 + \sigma_{123} x_6$, respectively. With the new basis, it is easy to verify that the differential agrees precisely with \eqref{eq:CFDm(V,Q)} (with primes added as appropriate).
\end{remark}

\begin{proposition} \label{prop:CFAm(V,Q)}
The type $A$ structure $\CFAm(V,Q)$ has a basis
\[
\{x_0, \dots, x_6, y_1, \dots, y_6\}
\]
with the following properties:

\begin{itemize}
\item
The generators $x_0, x_2, x_4, y_2, y_4$ are in $\CFAm(V,Q) \cdot \iota_0$, and the remaining ones are in $\CFAm(V,Q) \cdot \iota_1$.

\item
The $\AA_\infty$ multiplications are as follows:
\[
\xymatrix@R=0.6in@C=0.6in{
x_0 \ar[dr]_{U \rho_3} &
x_1 \ar[l]_{\rho_2} \ar[d]_(0.35){\substack{U^2 + \\ U\rho_{23}}} &
x_2 \ar[d]_{U} \ar[l]_{\rho_1} \ar@/_1.5pc/[ll]_{\rho_{12}} \ar[dl]_(0.4){U\rho_{123}} &
x_3 \ar[dl]^{\rho_2} \ar[d]^{U}  &
x_4 \ar[l]^{\rho_1} \ar[d]^{U} \ar@/_4pc/[lld]_(0.3){\rho_{12}} &
x_5 \ar[d]^{U} &
x_6 \ar[d]^{U} \\
& y_1 &
y_2 \ar[l]_{U \rho_1} &
y_3 \ar@/^1.5pc/[ll]^(0.4){\rho_2\rho_1} &
y_4 \ar[l]^{\rho_1} \ar@/^3pc/[lll]^{\rho_{12}\rho_1} &
y_5 &
y_6
}
\]
(In other words, we have $m_1(x_1) = U^2 y_1$, $m_2(x_1, \rho_{23}) = U y_1$, $m_3(y_3,\rho_2, \rho_1) = y_1$, and so on.)
\end{itemize}
\end{proposition}

\begin{proof}
This follows from a modified version of the algorithm described in \cite[Theorem 2.2]{HeddenLevineSplicing}, or by direct computation using the description of $\CFAA(\I)$ given in \cite[Section 10.1]{LOTBimodules}. (Note that we have reverted to referring to the algebra elements as $\rho_I$ rather than $\sigma_I$.)
\end{proof}

\begin{lemma} \label{lemma:CFK(Q(U))}
The constants associated to the generators of $\CFAm(V,Q) \cdot \iota_0$ via Proposition \ref{prop:Aglue} are $C_{x_0} = C_{x_2} = -2$, $C_{y_2} = C_{x_4} = -1$, and $C_{y_4}=0$.
\end{lemma}

\begin{proof}
Let $O \subset S^3$ denote the unknot and $X_O$ its complement, equipped with the $0$-framing. Note that $Q(O)$ is also the unknot. The type-$D$ structure $\CFD(X_O)$ has a single generator $\xi_0$, in Alexander grading $0$. The tensor product complex $\CFAm(V,Q) \boxtimes \CFD(X_O)$ is as follows:
\begin{equation} \label{eq:CFK(Q(U))}
\xymatrix{
 & x_2 \otimes \xi_0 \ar[dl] \ar[d]^U & x_4 \otimes \xi_0 \ar[dl] \ar[d]^U \\
x_0 \otimes \xi_0 & y_2 \otimes \xi_0 & y_4 \otimes \xi_0
}
\end{equation}
Note that $A_{Q(O)}(x_0 \otimes \xi_0) = A_{Q(O)}(x_2 \otimes \xi_0) = A_{Q(O)}(y_2 \otimes \xi_0) -1$ and $A_{Q(O)}(y_2 \otimes \xi_0) = A_{Q(O)}(x_4 \otimes \xi_0) = A_{Q(O)}(y_4 \otimes \xi_0) -1$. The homology is $\F[U]$, generated by $y_4 \otimes \xi_0$, which has Alexander grading $0$ since $\tau(Q(O))=0$. This determines the Alexander gradings of the remaining generators, and hence the constants from Proposition \ref{prop:Aglue}.
\end{proof}

\begin{proof}[Proof of \eqref{eq:tau(Q(K))}]
Consider the tensor product of $\CFAm(V,Q)$ with $\CFD(X_K)$, where $K$ is a knot in $S^3$ and $X_K$ is its exterior equipped with the $0$-framing. Note that the generators $x_5, x_5', x_6, x_6'$ do not affect $\tau(Q(K))$, since their tensor products with $\CFD(X_K)$ produce summands of $\CFKm(S^3,Q(K))$ that are $U$-torsion.

In the case where $\epsilon(K)=0$, $\CFD(X_K)$ has a summand isomorphic to $\CFD(X_O)$, so $\CFAm(V,Q) \boxtimes \CFD(X_K)$ has a summand isomorphic to \eqref{eq:CFK(Q(U))}. It follows immediately that $\tau(Q(K)) = \tau(Q(O)) = 0$. Thus, we restrict to the cases where $\epsilon(K)=\pm 1$. Let $s = 2\abs{\tau(K)}$.

Let $W$ be the subspace of $\CFK(S^3, Q(K))$ generated by the elements
\begin{equation} \label{eq:unstablegens}
\{x_0 \otimes \eta_0, x_2 \otimes \xi_0, y_2 \otimes \xi_0\} \cup \{x_1 \otimes \mu_1, y_1 \otimes \mu_i \mid i=1, \dots, s\} \cup H,
\end{equation}
where
\begin{equation} \label{eq:horizontalgens}
H =
\begin{cases}
\emptyset & \text{if } \epsilon(K) = -1 \\
\{x_3 \otimes \lambda, y_3 \otimes \lambda\} & \text{if } \epsilon(K)=1.
\end{cases}
\end{equation}
(In the latter case $\lambda = \lambda^1_{\ell_1}$ is the final element in the horizontal chain that terminates in $\xi_0 = \eta_2$, with a differential $\lambda \xrightarrow{D_2} \xi_0$.) Using Theorem \ref{thm:cfkcfd} and Proposition \ref{prop:CFAm(V,Q)}, it is not hard to verify that $W$ is a direct summand of $\CFK(S^3, Q(K))$ (as a chain complex); it has no incoming or outgoing differentials. We will see that the homology of $W$ contains a $\F[U]$ part, meaning that $\tau(Q(K))$ is determined completely by $W$.

Recall that $A_K(\xi_0) = \tau(K)$ and $A_K(\eta_0) = -\tau(K)$. By Proposition \ref{prop:Aglue} and Lemma \ref{lemma:CFK(Q(U))}, we have:
\begin{align*}
A_{Q(K)}(x_0 \otimes \eta_0) &= -A_K(\eta_0)-2 = \tau(K) -2 \\
A_{Q(K)}(x_2 \otimes \xi_0) &= -A_K(\xi_0)-2 = -\tau(K) -2 \\
A_{Q(K)}(y_2 \otimes \xi_0) &= -A_K(\xi_0)-1 = -\tau(K) -1.
\end{align*}

We consider three cases according to $\tau(K)$:

\begin{itemize}
\item
When $\tau(K)>0$, the unstable chain in $\CFD(X_K)$ (along with the possible $D_2$ differential into $\xi_0$ if $\epsilon(K)=1$) has the form
\[
\xymatrix{
\eta_0 \ar[r]^{D_3} & \mu_1 \ar[r]^{D_{23}} & \cdots \ar[r]^{D_{23}} & \mu_s & \xi_0 \ar[l]_{D_1} & \lambda. \ar@{-->}[l]_{D_2}
}
\]
The differential on $W$ is as follows:
\begin{equation} \label{eq:unstable-tau>0}
\xymatrix{
x_0 \otimes \eta_0 \ar[dr]^{U} & x_1 \otimes \mu_1 \ar[d]^{U^2} \ar[dr]^{U} & \dots \ar[dr]^{U} &
x_1 \otimes \mu_{s-1} \ar[d]^{U^2} \ar[dr]^{U} &
x_1 \otimes \mu_s \ar[d]^{U^2} & x_2 \otimes \xi_0 \ar[l] \ar[d]^{U}  \\
 & y_1 \otimes \mu_1 & \dots & y_1 \otimes \mu_{s-1} & y_1 \otimes \mu_s & y_2 \otimes \xi_0 \ar[l]_{U} \\
&&&& y_3 \otimes \lambda \ar@{-->}[u] & x_3 \otimes \lambda \ar@{-->}[l]_{U} \ar@{-->}[u]
}
\end{equation}
where the elements in the bottom row are included if $\epsilon(K)=1$.

In addition to the Alexander gradings compute above, we see inductively that
\begin{align*}
A_{Q(K)} (y_1 \otimes \mu_j) &= A_{Q(K)}(x_0 \otimes \eta_0) - j + 2 = \tau(K)-j \\
A_{Q(K)} (x_1 \otimes \mu_j) &= A_{Q(K)}(x_0 \otimes \eta_0) - j = \tau(K) - j -2 \\
\intertext{and, in the case where $\epsilon(K)=1$,}
A_{Q(K)} (y_3 \otimes \lambda) &= A_{Q(K)} (y_1 \otimes \mu_s) = -\tau(K) \\
A_{Q(K)} (x_3 \otimes \lambda) &= A_{Q(K)} (y_2 \otimes \xi_0) = -\tau(K) - 1.
\end{align*}
Regardless of $\epsilon(K)$, the element
\[
U^{s-1} x_0 \otimes \eta_0 + U^{s-2} x_1 \otimes \mu_1 + \cdots + x_1 \otimes \mu_{s-1} + y_2 \otimes \xi_0
\]
is a cycle in Alexander grading $-\tau(K)-1$ that generates the $\F[U]$-free part of the homology, and there is no such cycle in higher Alexander grading. Therefore, $\tau(Q(K)) = \tau(K)+1$.

\item
When $\tau(K)<0$, the unstable chain in $\CFD(X_K)$ has the form
\[
\xymatrix{
\lambda \ar@{-->}[r]^{D_2} & \xi_0 \ar[r]^{D_{123}} & \mu_1 \ar[r]^{D_{23}} & \cdots \ar[r]^{D_{23}} & \mu_s \ar[r]^{D_2} & \eta_0,
}
\]
where again $\lambda$ is included if $\epsilon(K)=1$. The differential on $W$ now takes the form:
\[
\xymatrix{
x_3 \otimes \lambda \ar@{-->}[d]^{U} \ar@{-->}[dr] &
x_2 \otimes \xi_0 \ar[dr]^{U} \ar[d]^{U} & x_1 \otimes \mu_1 \ar[d]^{U^2} \ar[dr]^{U} & \dots \ar[dr]^{U} & x_1 \otimes \mu_s \ar[d]^{U^2} \ar[dr] & \\
y_3 \otimes \lambda & y_2 \otimes \xi_0 & y_1 \otimes \mu_1 & \dots & y_1 \otimes \mu_s & x_0 \otimes \eta_0
}
\]
In this case, we have
\begin{align*}
A_{Q(K)} (x_1 \otimes \mu_j) &= A_{Q(K)}(x_0 \otimes \eta_0) + s - j  = -\tau(K) - j -2 \\
A_{Q(K)} (y_1 \otimes \mu_j) &= A_{Q(K)}(x_0 \otimes \eta_0) + s - j +2  = -\tau(K) - j
\intertext{and, in the case where $\epsilon(K)=1$,}
A_{Q(K)} (x_3 \otimes \lambda) &= A_{Q(K)} (y_2 \otimes \xi_0) = -\tau(K) - 1\\
A_{Q(K)} (y_3 \otimes \lambda) &= A_{Q(K)} (x_3 \otimes \lambda) + 1 = -\tau(K).
\end{align*}
A simple change of basis shows that the free part of the homology is generated by $y_2 \otimes \xi_0$ if $\epsilon(K)=-1$, but by $y_3 \otimes \lambda$ if $\epsilon(K)=1$. Thus, in this case,
\[
\tau(Q(K)) =
\begin{cases}
\tau(K) & \epsilon(K)=1 \\
\tau(K)+1 & \epsilon(K)=-1.
\end{cases}
\]

\item
When $\tau(K)=0$, the unstable chain in $\CFD(X_K)$ has the form
\[
\xymatrix{
\lambda \ar@{-->}[r]^{D_2} & \xi_0 \ar[r]^{D_{12}} & \eta_0,
}
\]
where again $\lambda$ is included if $\epsilon(K)=1$. The differential on $W$ is
\[
\xymatrix{
x_3 \otimes \lambda \ar@{-->}[d]^{U} \ar@{-->}[dr] &
x_2 \otimes \xi_0 \ar[dr] \ar[d]^{U} & \\
y_3 \otimes \lambda & y_2 \otimes \xi_0 & x_0 \otimes \eta_0.
}
\]
Just as in the previous case, we obtain
\[
\tau(Q(K)) =
\begin{cases}
0 & \epsilon(K)=1 \\
1 & \epsilon(K)=-1.
\end{cases} \qedhere
\]
\end{itemize}
\end{proof}

\section{Computation of \texorpdfstring{$\epsilon(Q(K))$}{\textepsilon(Q(K))}} \label{sec:epsilon}

Our next task is to compute $\epsilon(Q(K))$, completing the proof of Theorem \ref{thm:Q}. Specifically, we claim that
\begin{equation} \label{eq:epsilon(Q(K))}
\epsilon(Q(K)) = \begin{cases} 0 & \epsilon(K)=0 \\ 1 & \epsilon(K) \ne 0. \end{cases}
\end{equation}
To see this, we will directly compute the values of $\tau$ for the $(2,1)$ and $(2,-1)$ cables of $Q(K)$. For any knot $J$, let $J_{p,q}$ denote the $(p,q)$ cable of $J$. Hom \cite{HomTau} found a formula for $\tau(J_{p,q})$ in terms of $\tau(J)$ and $\epsilon(J)$. Specifically, in the cases where $p=2$ and $q=\pm1$, the formula states:
\begin{theorem} \label{thm:hom}
For any knot $J \subset S^3$, we have:
\begin{align}
\label{eq:J21}
\tau(J_{2,1}) &=
\begin{cases}
2 \tau(J) +1 & \text{if } \epsilon(J) = -1 \\
0 & \text{if } \epsilon(J)=0 \\
2 \tau(J) & \text{if }\epsilon(J) = 1.
\end{cases} \\
\label{eq:J2-1}
\tau(J_{2,-1}) &=
\begin{cases}
2 \tau(J)  & \text{if } \epsilon(J) = -1 \\
0 & \text{if } \epsilon(J)=0 \\
2 \tau(J) -1 & \text{if }\epsilon(J) = 1.
\end{cases}
\end{align}
In particular, $\epsilon(J) = -1$ if and only if $\tau(J_{2,1})$ is odd, and $\epsilon(J)=1$ if and only if $\tau(J_{2,-1})$ is odd.
\end{theorem}

The first case in \eqref{eq:epsilon(Q(K))} is straightforward: If $K$ is any knot with $\epsilon(K)=0$, the $\tau$ invariant of any satellite of $K$ is the same as the $\tau$ invariant of the corresponding satellite of the unknot $O$, since $\CFD(X_K)$ has a summand that is isomorphic to $\CFD(X_O)$. (See \cite[Section 4.3]{HomTau}.) In particular, $\tau(Q(K)_{2,1} =\tau(Q(K)_{2,-1}) = 0$, which implies that $\epsilon(Q(K))=0$.

For knots with $\epsilon(K) \ne 0$, the bulk of \eqref{eq:epsilon(Q(K))} follows from the following proposition:

\begin{proposition} \label{prop:not-1}
For any knot $K \subset S^3$, we have $\tau(Q(K)_{2,1}) = 2 \tau(Q(K))$. Therefore, $\epsilon(Q(K)) \ne -1$; whenever $\tau(Q(K)) \ne 0$, $\epsilon(Q(K))=1$.
\end{proposition}

Note that Proposition \ref{prop:not-1} suffices for the proof of Theorem \ref{thm:nonsurjective}, which in turn implies Theorem \ref{thm:main}. The only remaining cases in \eqref{eq:epsilon(Q(K))} are those where $\tau(Q(K))=0$, in which case the fact that $\tau(Q(K)_{2,1})=0$ does not determine whether $\epsilon(Q(K))=0$ or $1$. These cases are treated as follows:

\begin{proposition} \label{prop:finalcases}
For any knot $K \subset S^3$ with either $\tau(K)=-1$ and $\epsilon(K)=-1$, or $\tau(K)=0$ and $\epsilon(K)=1$, we have $\tau(Q(K)_{2,-1}) = 1$. Therefore, $\epsilon(Q(K))=1$.
\end{proposition}

We shall provide a detailed proof of Proposition \ref{prop:not-1}, and then sketch the modifications needed for Proposition \ref{prop:finalcases}.

To prove Proposition \ref{prop:not-1}, let $C_{2,1}$ denote a $(2,1)$ curve in the solid torus $V$, represented by the Heegaard diagram in Figure \ref{fig:cables}(b), and let $Q_{2,1} \subset V$ be the $(2,1)$ cable of $Q$, obtained by gluing $(V,C_{2,1})$ to the $2$-bridge link component $X(L)$. The gluing theorem states that
\[
\CFDm(V,Q_{2,1}) \simeq \CFAm(V,C_{2,1}) \boxtimes_{\AA_\rho} \CFDD(X(L)).
\]
As in the previous section, this describes the orientation of $Q_{2,1}$ whose winding number is $-2$.

According to Hom \cite[Section 4.1]{HomTau}, we have:

\begin{lemma} \label{lemma:CFAm(V,C21)}
The type $A$ module $\CFAm(V,C_{2,1})$ has generators $a,b,c$, with $\AA_\infty$ multiplications given by:
\[
\begin{aligned}
m_{3+i}(a, \rho_3, \underbrace{\rho_{23}, \dots, \rho_{23}}_{i}, \rho_2) &= U^{2i+2} \cdot a \text{ for all } i \ge 0 \\
m_{4+i}(a, \rho_3, \underbrace{\rho_{23}, \dots, \rho_{23}}_{i}, \rho_2, \rho_1) &= U^{2i+1} \cdot b \text{ for all } i \ge 0 \\
m_2(a, \rho_1) &= c \\
m_1(b) &= U \cdot c.
\end{aligned}
\]
\end{lemma}

\begin{figure}[t]
\[
\xymatrix{
p_1 \ar[d]^{U^2} & p_2 \ar[d]^{U} \ar[l]_{\sigma_1} \ar[r]^{\sigma_3}
& p_3 \ar[d]^{U} \ar[r]^{\sigma_2} & p_4 \ar[d]^{U} \ar[r]^{\sigma_1} & p_5 \ar[d]^{U}
& x_0 \ar[l]^{U\sigma_{123}} \ar[dr]|{U^2 \sigma_1} & x_1 \ar[d]|{U^4} \ar[l]_{\sigma_2} \ar@/_1.5pc/[ll]_{U^3 \sigma_{23}}
& x_2 \ar[d]|{U^2} \ar[l]_{\sigma_3} \ar@/_4pc/[lllllll]_{U\sigma_{123}}
& x_3 \ar[d]|(0.45){U^2} \ar[dl]|{\sigma_2} & x_4 \ar[d]^{U^2} \ar[l]_{\sigma_3} \ar[ldddd]_(0.8){U\sigma_3} \ar@/^1pc/[ddddlllllll]_(0.6){U\sigma_{123}} & \\
q_1 & q_2 \ar[l]_{U \sigma_1} \ar[r]_{\sigma_3}
& q_3 \ar[r]_{\sigma_2} & q_4 \ar[r]_{\sigma_1} & q_5 &
& y_1 \ar[ll]_{\sigma_{23}} & y_2 \ar[l]_{U^2 \sigma_3} \ar@/^2.5pc/[lllllll]_{U\sigma_{123}}
& y_3 \ar@/^2pc/[ll]^{\sigma_{23}} & y_4 \ar[l]_{\sigma_3} \ar@/^1pc/[ddddlllllll]_(0.4){U\sigma_{123}} \ar[ldddd]^{\sigma_3}
& \\
\\ \\
& & r_1 \ar[d]^{U^2} & r_2 \ar[d]|(0.4){U} \ar[l]^{\sigma_1} \ar[r]^{\sigma_3}
& r_3 \ar[d]^{U} \ar[r]^(0.4){\sigma_2} & r_4 \ar[d]^{U} \ar[r]^{\sigma_1} & r_5 \ar[d]^{U}
& r_6 \ar[d]_{U} \ar[l]^{\sigma_{123}} \ar@/^1pc/[uuulllllll]^(0.6){\sigma_{123}} & r_7 \ar[d]_{U} \ar[l]_{\sigma_2}
& r_8 \ar[d]^{U} \ar[l]_(0.4){\sigma_{23}} & r_9 \ar[d]^{U} \ar[l]_{\sigma_3} \\
& & s_1 & s_2 \ar[l]^{U \sigma_1} \ar[r]_{\sigma_3} & s_3 \ar[r]_{\sigma_2} & s_4 \ar[r]_{\sigma_1} & s_5
& s_6 \ar[l]^{\sigma_{123}} & s_7 \ar[l]^{\sigma_2} & s_8 \ar[l]^{\sigma_{23}} & s_9 \ar[l]^{\sigma_3} \\
}
\]
\caption{A summand of $\CFDm(V,Q_{2,1})$.} \label{fig:CFDm(V,Q21)}
\end{figure}

\begin{proposition} \label{prop:CFDm(V,Q21)}
\allowdisplaybreaks
The type $D$ structure $\CFDm(V,Q_{2,1})$ has a direct summand with a basis
\[
\{p_1, \dots, p_5, q_1, \dots, q_5, r_1, \dots, r_9, s_1, \dots, s_9, x_0, \dots, x_4, y_1, \dots, y_4\}
\]
with the following properties:
\begin{itemize}
\item
The generators
\[
p_2, p_4, q_2, q_4, r_2, r_4, r_6, r_9, s_2, s_4, s_6, s_9, x_0, x_2, y_2, x_4, y_4
\]
are in $\iota_0 \CFDm(V,Q_{2,1})$, and the remaining ones are in $\iota_1 \CFDm(V,Q_{2,1})$.

\item
The differential is as shown in Figure \ref{fig:CFDm(V,Q21)}.
\end{itemize}
\end{proposition}

\begin{proof}
\allowdisplaybreaks
In the tensor product $\CFAm(V,C_{2,1}) \boxtimes_{\AA_\rho} \CFDD(X(L))$, we will denote generators $a \otimes g_i$, $b \otimes g_i$, or $c \otimes g_i$ by $a_i, b_i, c_i$ respectively; there are $55$ generators in all. As in the previous section, it is not hard to verify that the differential is as follows (where terms with coefficient $1$ are indicated in boldface):
\begin{align*}
d(a_1) &= \bm{c_{24}} && \\
d(a_2) &= \sigma_{123} c_{12} + U^2 \sigma_3 a_{30} && \\
d(b_3) &= U c_3 &
d(c_3) &= 0 \\
d(b_4) &= U c_4 + \sigma_3 b_{26} &
d(c_4) &= \sigma_3 c_{26} \\
d(b_5) &= U c_5 + \sigma_1 b_{31} &
d(c_5) &= \sigma_1 c_{31} \\
d(b_6) &= U c_6 + \sigma_2 b_{25} &
d(c_6) &= \sigma_2 c_{25} \\
d(a_7) &= U^2 a_1 + \bm{c_{12}} + U b_{24} + \sigma_{23} c_{31} && \\
d(b_8) &= U c_8 &
d(c_8) &= 0 \\
d(b_9) &= U c_9 &
d(c_9) &= 0 \\
d(b_{10}) &= U c_{10} + \sigma_1 b_{12} + \sigma_{123} b_{31} + \sigma_3 b_{33} &
d(c_{10}) &= \sigma_1 c_{12} + \sigma_{123} c_{31} + \sigma_3 c_{33} & \\
d(a_{11}) &= U b_{14} + \bm{c_{17}} + U^2 a_{19} + \sigma_{23} c_{24}  && \\
d(b_{12}) &= U c_{12} &
d(c_{12}) &= 0 \\
d(a_{13}) &= \sigma_2 a_{27} + U^4 a_{30} + U^3\sigma_{23} b_{31} && \\
d(b_{14}) &= U c_{14} &
d(c_{14}) &= 0 \\
d(a_{15}) &= \sigma_2 c_{25} + \sigma_{23} a_{30} && \\
d(b_{16}) &= \sigma_{123} b_8 + \sigma_1 b_{17} + U c_{16} + \sigma_3 b_{22}  &
d(c_{16}) &= \sigma_{123} c_8 + \sigma_1 c_{17} + \sigma_3 c_{22} \\
d(b_{17}) &= U c_{17} &
d(c_{17}) &= 0 \\
d(a_{18}) &= \sigma_2 a_2 + U^2 a_{15} + U \sigma_2 b_{25} &&  \\
d(a_{19}) &= \bm{c_{14}} && \\
d(b_{20}) &= \sigma_2 b_9 +  U c_{20} &
d(c_{20}) &= \sigma_2 c_9 \\
d(a_{21}) &= \sigma_3 a_{15} + \sigma_{123} c_{17} && \\
d(b_{22}) &= U c_{22} + \sigma_2 b_{23}  &
d(c_{22}) &= \sigma_2 c_{23} \\
d(b_{23}) &= \sigma_1 b_8 + U c_{23} &
d(c_{23}) &= \sigma_1 c_8 \\
d(b_{24}) &= U c_{24} &
d(c_{24}) &= 0 \\
d(b_{25}) &= \sigma_{123} b_{24} + \sigma_{123} b_8 + U c_{25} &
d(c_{25}) &= \sigma_{123} c_{24} + \sigma_{123} c_8 \\
d(b_{26}) &= \sigma_{23} b_6 + U c_{26} &
d(c_{26}) &= \sigma_{23} c_6 \\
d(a_{27}) &= U^2 \sigma_1 a_{30} + U \sigma_{123} b_{31} && \\
d(b_{28}) &= U c_{28} &
d(c_{28}) &= 0 \\
d(a_{29}) &= U \sigma_{123} b_{17} + \sigma_3 a_{18} + U^2 a_{21} && \\
d(a_{30}) &= \sigma_{23} c_{31} && \\
d(b_{31}) &= U c_{31} &
d(c_{31}) &= 0 \\
d(b_{32}) &= \sigma_3 b_{20} +  U c_{32} &
d(c_{32}) &= \sigma_3 c_{20} \\
d(b_{33}) &= \sigma_2 b_5 +  U c_{33} &
d(c_{33}) &= \sigma_2 c_5 \\
d(a_{34}) &= U^2 a_2 + \sigma_3 a_{13} + U \sigma_{123} b_{12} &&
\end{align*}

First, observe that the sets $\{b_9, c_9, b_{20}, c_{20}, b_{32}, c_{32}\}$, $\{b_3,c_3\}$, and $\{b_{28}, c_{28}\}$ each generate isolated summands whose tensor products with $\CFD(X_K)$ will always be torsion $\F[U]$-modules. These summands will not affect $\tau(Q(K)_{2,1})$, so we may disregard them.

Next, we perform a change of basis that cancels the four terms indicated in boldface above and further simplifies the differential. Define:
\begin{align*}
a_{2}' &= a_2 + \sigma_{123} a_7 &
c_{10}' &= c_{10} + \sigma_1 a_7 \\
a_{11}' &= a_{11} + \sigma_{23} a_1 &
b_{12}' &= b_{12} + U a_7 + \sigma_{23} b_{31} \\
c_{12}' &= c_{12} + U^2 a_1 + U b_{24} + \sigma_{23} c_{31} &
b_{14}' &= b_{14} + U a_{19} \\
a_{15}' &= a_{15} + c_6 &
b_{16}' &= b_{16} + b_{25} \\
c_{16}' &= c_{16} + \sigma_1 a_{11} + c_{25} &
b_{17}' &= b_{17} + U a_{11} + \sigma_{23} b_{24} \\
c_{17}' &= c_{17} + U b_{14} + U^2 a_{19} &
a_{18}' &= a_{18} + U b_6 \\
a_{21}' &= a_{21} + \sigma_{123} a_{11}  &
b_{24}' &= b_{24} + U a_1 \\
c_{25}' &= c_{25} + \sigma_{123} a_1 .
\end{align*}
Note that $d(a_7) = c_{12}'$ and $d(a_{11}') = c_{17}'$. The sets $\{a_1, c_{24}\}$, $\{a_{19}, c_{14}\}$, $\{a_{11}', c_{12}'\}$, and $\{a_{11}', c_{17}'\}$ generate acyclic summands that can be canceled. The differential applied to the remaining primed generators is as follows:
\begin{align*}
d(a_{2}') &= U \sigma_{123} b_{24}' + U^2 \sigma_3 a_{30} &
d(c_{10}') &= U \sigma_1 b_{24}' + \sigma_3 c_{33} \\
d(b_{12}') &= U^2 b_{24}' &
d(b_{14}') &= 0 \\
d(a_{15}') &= \sigma_{23} a_{30} &
d(b_{16}') &= U c_{16}' + \sigma_1 b_{17}' + \sigma_3 b_{22} \\
d(c_{16}') &= U \sigma_1 b_{14}' + \sigma_3 c_{22} &
d(b_{17}') &= U^2 b_{14}' \\
d(a_{18}') &= \sigma_2 a_2' + U^2 a_{15}' &
d(a_{21}') &= U \sigma_{123} b_{14}' + \sigma_3 a_{15} \\
d(b_{24}') &= 0  &
d(c_{25}') &= \sigma_{123} c_8.
\end{align*}
Additionally, in the new basis, we have:
\begin{align*}
d(b_{10}) &= U c_{10}' + \sigma_1 b_{12}' + \sigma_3 b_{33} \\
d(b_{25}) &= \sigma_{123} b_8 + \sigma_{123} b_{24}' + U c_{25}' \\
d(a_{29}) &= U \sigma_3 b_6 + U \sigma_{123} b_{17}' + \sigma_3 a_{18} + U^2 a_{21} \\
d(a_{34}) &= U^2 a_2' + U \sigma_{123} b_{12}' + \sigma_3 a_{13}.
\end{align*}

We now rename the generators as follows:
\begin{align*}
p_1 &= b_{12}'  & p_2 &= b_{10}   & p_3 &= b_{33}  & p_4 &= b_5  &   p_5 &= b_{31} \\
q_1 &= b_{24}'  & q_2 &= c_{10}'  & q_3 &= c_{33}  & q_4 &= c_5  &   q_5 &= c_{31} \\
x_0 &= a_{27}   & x_1 &= a_{13}   & x_2 &= a_{34}  & x_3 &= a_{18}' & x_4 &= a_{29} \\
&               & y_1 &= a_{30}   & y_2 &= a_2'    & y_3 &= a_{15}' & y_4 &= a_{21}' \\
r_1 &= b_{17}'  & r_2 &= b_{16}'  & r_3 &= b_{22}  & r_4 &= b_{23} & r_5 &= b_8 \\
r_6 &= b_{25}   & r_7 &= b_6      & r_8 &= b_{26}  & r_9 &= b_4  & \\
s_1 &= b_{14}'  & s_2 &= c_{16}'  & s_3 &= c_{22}  & s_4 &= c_{23} & s_5 &= c_8 \\
s_6 &= c_{25}   & s_7 &= c_6      & s_8 &= c_{26}  & s_9 &= c_4.  &
\end{align*}
It is simple but tedious to verify that the differential on these generators agrees with the statement of the theorem.
\end{proof}

\begin{figure}
\[
\xymatrix@R=0.6in@C=0.52in{
p_1 \ar[d]|{U^2}
& p_2 \ar[d]|{U} \ar[l]|{\rho_3} \ar[r]|{\rho_1} \ar@/^1pc/[rr]|{\rho_{12}} \ar@/^2pc/[rrr]|{\rho_{123}}
& p_3 \ar[d]|{U} \ar[r]|{\rho_2} \ar@/^1pc/[rr]|{\rho_{23}}
& p_4 \ar[d]|{U} \ar[r]|{\rho_3} & p_5 \ar[d]^{U}
& x_0 \ar[l]|{U\rho_3 \rho_2 \rho_1} \ar[dr]|{U^2 \rho_3}
& x_1 \ar[d]|{\substack{U^4 + \\ U^2 \rho_{23}}} \ar[l]|{\rho_2}
 \ar@/_1.5pc/[ll]|{\substack{U^3 \rho_2 \rho_1 + \\ U \rho_{23} \rho_2 \rho_1 }}
& x_2 \ar[d]|{U^2} \ar[l]|{\rho_1} \ar[dl]|{U^2 \rho_{123}}
 \ar@/_1.5pc/[ll]|{\rho_{12}} \ar@/_3pc/[lll]|{\substack{U^3 \rho_{12} \rho_1 + \\ U \rho_{123} \rho_2 \rho_1 }}
 \ar@/^1.5pc/[lllllll]|{U\rho_3 \rho_2 \rho_1}
& & & \\
q_1 & q_2 \ar[l]|(0.35){U \rho_3} \ar[r]|{\rho_1} \ar@/_1pc/[rr]|{\rho_{12}} \ar@/_2pc/[rrr]|{\rho_{123}}
& q_3 \ar[r]|{\rho_2} \ar@/_1pc/[rr]|{\rho_{23}} & q_4 \ar[r]|{\rho_3} & q_5 &
& y_1 \ar[ll]|{\rho_2 \rho_1}
& y_2 \ar[l]|(0.55){U^2 \rho_1} \ar@/^1.5pc/[lll]|(0.65){U^2 \rho_{12} \rho_1} \ar@/^4pc/[lllllll]|(0.3){U\rho_3 \rho_2 \rho_1} & & & \\
&&&&&&& x_3 \ar[d]|{U^2} \ar[u]|{\rho_2}
& x_4 \ar[l]|{\rho_1} \ar[d]|(0.45){U^2} \ar[lu]|{\rho_{12}}  \ar@/_3pc/[ddddllllllll]|(0.7){U\rho_3 \rho_2 \rho_1}
\ar@/^1pc/[ddddll]|(0.15){U \rho_1} \ar[ddddlll]|(0.4){U \rho_{12}} \ar@/_1pc/[ddddllll]|{U \rho_{123} \rho_2 \rho_1}\\
&&&&&&& y_3 \ar@/^1pc/[luu]|{\rho_2 \rho_1}
& y_4 \ar[l]|{\rho_1} \ar@/^1pc/[lluu]|(0.6){\rho_{12} \rho_1} \ar@/_3pc/[ddddllllllll]|(0.6){U\rho_3 \rho_2 \rho_1}
\ar@/^1pc/[ddddll]|(0.3){\rho_1} \ar[ddddlll]|(0.15){\rho_{12}} \ar@/_1pc/[ddddllll]|(0.25){\rho_{123} \rho_2 \rho_1} \\
\\
\\
r_1 \ar[d]|{U^2}
& r_2 \ar[d]|{U} \ar[l]|{\rho_3} \ar[r]|{\rho_1} \ar@/^1pc/[rr]|{\rho_{12}} \ar@/^2pc/[rrr]|{\rho_{123}}
& r_3 \ar[d]|{U} \ar[r]|{\rho_2} \ar@/^1pc/[rr]|{\rho_{23}}
& r_4 \ar[d]|{U} \ar[r]|{\rho_3}
& r_5 \ar[d]|{U}
& r_6 \ar[d]|{U} \ar[l]|{\rho_3 \rho_2 \rho_1} \ar[uuuuulllll]|(0.36){\rho_3 \rho_2 \rho_1}
& r_7 \ar[d]|{U} \ar[l]|{\rho_2} \ar[uuuuullllll]|(0.39){\rho_{23} \rho_2 \rho_1}
& r_8 \ar[d]|{U} \ar[l]|{\rho_2 \rho_1} \ar[uuuuulllllll]|(0.42){\rho_2 \rho_{123} \rho_2 \rho_1}
& r_9 \ar[d]|{U} \ar[l]|{\rho_1} \ar[uuuuullllllll]|(0.45){\rho_{12} \rho_{123} \rho_2 \rho_1}
\\
s_1
& s_2 \ar[l]|{U \rho_3} \ar[r]|{\rho_1} \ar@/_1pc/[rr]|{\rho_{12}} \ar@/_2pc/[rrr]|{\rho_{123}}
& s_3 \ar[r]|{\rho_2} \ar@/_1pc/[rr]|{\rho_{23}}
& s_4 \ar[r]|{\rho_3} & s_5
& s_6 \ar[l]|{\rho_3 \rho_2 \rho_1}
& s_7 \ar[l]|{\rho_2} \ar@/^1pc/[ll]|{\rho_{23} \rho_2 \rho_1}
& s_8 \ar[l]|{\rho_2 \rho_1} \ar@/^1pc/[ll]|{\rho_2 \rho_{12}} \ar@/^2pc/[lll]|{\rho_2 \rho_{123} \rho_2 \rho_1}
& s_9 \ar[l]|{\rho_1} \ar@/^1pc/[ll]|{\rho_{12} \rho_1} \ar@/^2pc/[lll]|{\rho_{12} \rho_{12}} \ar@/^3pc/[llll]|{\rho_{12} \rho_{123} \rho_2 \rho_1}
}
\]
\caption{A summand of $\CFAm(V, Q_{2,1})$. The only multiplications not shown are the higher multiplications $r_i \to r_j$ for $5 \le j < i \le 9$, which are the same as the ones $s_i \to s_j$. } \label{fig:CFAm(V,Q21)}
\end{figure}

Once again, we may tensor with the identity $AA$ bimodule to obtain:

\begin{proposition} \label{prop:CFAm(V,Q21)}
The type $A$ structure $\CFAm(V,Q_{2,1})$ has a summand with a basis
\[
\{p_1, \dots, p_5, q_1, \dots, q_5, r_1, \dots, r_9, s_1, \dots, s_9, x_0, \dots, x_4, y_1, \dots, y_4\}
\]
whose associated idempotents are just as in Proposition \ref{prop:CFDm(V,Q21)}, and whose $\AA_\infty$ multiplications are as shown in Figure \ref{fig:CFAm(V,Q21)}. \qed
\end{proposition}

Just as in the previous section, we note that $Q_{2,1}(O)$ is the unknot, and use this fact to pin down the absolute Alexander grading via Proposition \ref{prop:Aglue}. The tensor product $\CFAm(V,Q_{2,1}) \boxtimes \CFD(X_O)$ is
\begin{equation} \label{eq:CFK(Q21(U))}
\xymatrix{
 & x_2 \otimes \xi_0 \ar[dl] \ar[d]^{U^2} & x_4 \otimes \xi_0 \ar[dl] \ar[d]^{U^2} \ar[r]^{U} & r_6 \otimes \xi_0 \ar[d]^{U} & r_9 \otimes \xi_0 \ar[d]^{U} \ar[l] \\
x_0 \otimes \xi_0 & y_2 \otimes \xi_0 & y_4 \otimes \xi_0 \ar[r] & s_6 \otimes \xi_0 & s_9 \otimes \xi_0 \ar[l]
}
\end{equation}
plus summands that do not affect $\tau$. The homology is generated by $(y_4 + s_9) \otimes \xi_0$, which must have Alexander grading $0$. As in Lemma \ref{lemma:CFK(Q(U))}, we conclude that $C_{x_0} = C_{x_2} = -4$, $C_{y_2} = C_{y_4} = -2$, and $C_{y_4} = 0$.

\begin{proof}[Proof of Proposition \ref{prop:not-1}]
Let $K \subset S^3$ be any knot with $\epsilon(K) \ne 0$. (The case where $\epsilon(K)=0$ was discussed earlier.) As in the previous section, we consider a subspace of $\CFK(S^3, Q(K)_{2,1})$ generated by certain elements arising from the unstable chain in $\CFD(X_K)$ and ``nearby'' elements. Specifically, if $\epsilon(K) =1$, then let $\lambda = \lambda^1_{\ell_1}$ as in the previous section. If $\epsilon(K)=-1$, then set $\kappa = \kappa^1_{k_1}$, and $\kappa' = \kappa^1_{k_1-1}$ if $k_1>1$ (where $k_1$ is the length of the corresponding vertical arrow in $\CFKm(S^3,K)$). The vertical chain ending in $\xi_1 = \eta_0$ includes differentials
\[
\xi_2 \xrightarrow{D_{123}} \kappa \xleftarrow{D_1} \xi_1 = \eta_0
\]
if the corresponding vertical arrow in $\CFKm(S^3,K)$ has length $k_1 = 1$, or
\[
\kappa' \xrightarrow{D_{23}} \kappa \xleftarrow{D_1} \xi_1 = \eta_0
\]
if $k_1>1$.

Let $W$ be the subspace of $\CFK(S^3, Q(K)_{2,1})$ generated by the elements
\begin{equation} \label{eq:unstablegens-Q21}
\{x_0 \otimes \eta_0, x_2 \otimes \xi_0, y_2 \otimes \xi_0\} \cup \{x_1 \otimes \mu_1, y_1 \otimes \mu_i \mid i=1, \dots, s\} \cup H,
\end{equation}
where
\begin{equation} \label{eq:nearbygens-Q21}
H =
\begin{cases}
\{x_3 \otimes \lambda, y_3 \otimes \lambda\} & \text{if } \epsilon(K)=1 \\
\{p_2 \otimes \xi_2, q_2 \otimes \xi_2, p_5 \otimes \kappa, q_5 \otimes \kappa\}
 & \text{if } \tau(K) \le 0, \, \epsilon(K) = -1, \text{ and } k_1 = 1 \\
\{p_3 \otimes \kappa', q_3 \otimes \kappa', p_5 \otimes \kappa, q_5 \otimes \kappa\}
 & \text{if } \tau(K) \le 0, \, \epsilon(K) = -1, \text{ and } k_1 > 1.
\end{cases}
\end{equation}
The verification that $W$ is a direct summand is slightly trickier than in the previous section: Lemma \ref{lemma:D1D2D3(xi0)} guarantees that the multiplications $m_4(x_2, \rho_3, \rho_2, \rho_1) = U \cdot p_1$ and $m_4(y_2, \rho_3, \rho_2, \rho_1) = U \cdot q_1$ do not produce differentials from elements in $W$ to elements not in $W$, and Lemma \ref{lemma:D3(xi2)} does the same for the multiplications $m_2(p_2, \rho_3) = p_1$ and $m_2(q_2,\rho_3) = U \cdot q_1$ in the case where $k_1=1$. By Proposition \ref{prop:Aglue}, the Alexander gradings of certain of these generators are
\begin{align*}
A_{Q(K)_{2,1}}(x_0 \otimes \eta_0) &= -2A_K(\eta_0)-4 = 2\tau(K) -4 \\
A_{Q(K)_{2,1}}(x_2 \otimes \xi_0) &= -2A_K(\xi_0)-4 = -2\tau(K) -4 \\
A_{Q(K)_{2,1}}(y_2 \otimes \xi_0) &= -2A_K(\xi_0)-2 = -2\tau(K) -2.
\end{align*}
As in the previous section, we shall see that the free part of $\HFKm(S^3,K)$ is supported in $W$.

\begin{itemize}
\item
When $\tau(K)>0$, the differential on $W$ is as follows:
\begin{equation} \label{eq:unstable21-tau>0}
\xymatrix{
x_0 \otimes \eta_0 \ar[dr]^{U^2} & x_1 \otimes \mu_1 \ar[d]^{U^4} \ar[dr]^{U^2} & \dots \ar[dr]^{U^2} &
x_1 \otimes \mu_{s-1} \ar[d]^{U^4} \ar[dr]^{U^2} &
x_1 \otimes \mu_s \ar[d]^{U^4} & x_2 \otimes \xi_0 \ar[l] \ar[d]^{U^2}  \\
 & y_1 \otimes \mu_1 & \dots & y_1 \otimes \mu_{s-1} & y_1 \otimes \mu_s & y_2 \otimes \xi_0 \ar[l]_{U^2} \\
&&&& y_3 \otimes \lambda \ar@{-->}[u] & x_3 \otimes \lambda \ar@{-->}[l]_{U^2} \ar@{-->}[u] }
\end{equation}
where the generators in the bottom row are included if $\epsilon(K)=1$. Note that \eqref{eq:unstable21-tau>0} is identical to \eqref{eq:unstable-tau>0}, except that $U$ has been replaced by $U^2$ throughout. Just as in the previous section, the free part of the homology is generated by
\[
U^{s-1} x_0 \otimes \eta_0 + U^{s-2} x_1 \otimes \mu_1 + \cdots + x_1 \otimes \mu_{s-1} + y_2 \otimes \xi_0,
\]
which has Alexander grading $-2\tau(K) - 2$, so $\tau(Q(K)_{2,1}) = 2\tau(K) + 2 = 2 \tau(Q(K))$. Therefore, $\epsilon(Q(K)) = 1$.

\item
When $\tau(K)<0$, we consider the three cases in \eqref{eq:nearbygens-Q21}.

\begin{itemize}
\item If $\epsilon(K)=1$, the differential on $W$ takes the form:
\begin{equation} \label{eq:unstable21-tau<0,epsilon=1}
\xymatrix{
x_3 \otimes \lambda \ar[d]^{U^2} \ar[dr] &
x_2 \otimes \xi_0 \ar[dr]^{U^2} \ar[d]^{U^2} & x_1 \otimes \mu_1 \ar[d]^{U^4} \ar[dr]^{U^2} & \dots \ar[dr]^{U^2} & x_1 \otimes \mu_s \ar[d]^{U^4} \ar[dr] & \\
y_3 \otimes \lambda & y_2 \otimes \xi_0 & y_1 \otimes \mu_1 & \dots & y_1 \otimes \mu_s & x_0 \otimes \eta_0
}
\end{equation}
The free part of the homology is generated by $y_3 \otimes \lambda$. Since
\begin{align*}
A_{Q(K)_{2,1}}(y_3 \otimes \lambda) &= A_{Q(K)_{2,1}}(x_3 \otimes \lambda) + 2 \\
&= A_{Q(K)_{2,1}}(y_2 \otimes \xi_0) + 2 \\
&= -2\tau(K),
\end{align*}
we see that $\tau(Q(K)_{2,1}) = 2 \tau(K) = 2\tau(Q(K))$, and hence $\epsilon(Q(K))=1$.

\item If $\epsilon(K)=-1$ and $k_j=1$, the differential on $W$ takes the form:
\begin{equation} \label{eq:unstable21-tau<0,epsilon=-1}
\xymatrix@C=0.25in{
x_2 \otimes \xi_0 \ar[dr]^{U^2} \ar[d]^{U^2} & x_1 \otimes \mu_1 \ar[d]^{U^4} \ar[dr]^{U^2} & \dots \ar[dr]^{U^2} & x_1 \otimes \mu_s \ar[d]^{U^4} \ar[dr] \ar[rr]^{U^3} & &
 p_5 \otimes \kappa \ar[d]^{U} & p_2 \otimes \xi_2 \ar[l] \ar[d]^{U} \\
y_2 \otimes \xi_0 & y_1 \otimes \mu_1 & \dots & y_1 \otimes \mu_s \ar@/_1.5pc/[rr]_{} & x_0 \otimes \eta_0 &
 q_5 \otimes \kappa & q_2 \otimes \xi_2 \ar[l]
}
\end{equation}
The same is true when $k_j>1$, except that $p_2 \otimes \xi_2$ and $q_2 \otimes \xi_2$ are replaced with $p_3 \otimes \kappa'$ and $q_3 \otimes \kappa'$, respectively. In this case, the free part of the homology is generated by $y_2 \otimes \xi_0$, with Alexander grading $-2\tau(K) -2$. Therefore, $\tau(Q(K)_{2,1}) = 2\tau(K) + 2 = 2 \tau(Q(K))$. If $\tau(K)<-1$, we may conclude that $\epsilon(Q(K)) =1$; if $\tau(K)=1$, we merely see that $\epsilon(Q(K)) \ne -1$.
\end{itemize}

\item
When $\tau(K)=0$, an analysis similar to the previous case shows again that
\[
\tau(Q(K)_{2,1} = \begin{cases} 0 & \text{if } \epsilon(K)=1 \\ 2 & \text{if } \epsilon(K)=-1. \end{cases}
\]
In either case, $\tau(Q(K)_{2,1}) = 2\tau(Q(K))$. \qedhere
\end{itemize}
\end{proof}

Finally, we sketch the proof of Proposition \ref{prop:finalcases}. Let $C_{2,-1} \subset V$ denote the pattern for $(2,-1)$ cabling.

\begin{lemma} \label{lemma:CFAm(V,C2-1)}
The type $A$ module $\CFAm(V,C_{2,-1})$ has generators $a,b,c,d,e$, with $\AA_\infty$ multiplications given by:
\begin{align*}
m_2(a, \rho_1) &= c &
m_2(a, \rho_{12}) &= e \\
m_1(b) &= U \cdot c &
m_2(b, \rho_2) &= d \\
m_2(c, \rho_2) &= e &
m_1(d) &= U \cdot e
\end{align*}
\begin{align*}
m_{3+i}(a, \rho_3, \underbrace{\rho_{23}, \dots, \rho_{23}}_{i}, \rho_2) &= U^{2i+2} \cdot a \text{ for all } i \ge 0 \\
m_{4+i}(a, \rho_3, \underbrace{\rho_{23}, \dots, \rho_{23}}_{i}, \rho_2, \rho_1) &= U^{2i+1} \cdot b \text{ for all } i \ge 0 \\
m_{4+i}(a, \rho_3, \underbrace{\rho_{23}, \dots, \rho_{23}}_{i}, \rho_2, \rho_{12}) &= U^{2i+1} \cdot d \text{ for all } i \ge 0
\end{align*}
\end{lemma}

\begin{proof}
This can be seen directly from a Heegaard diagram or by taking the tensor product of $\CFAm(V,C_{2,-1})$ with the bimodule $\CFDA(\tau_m^{-1})$ from \cite[Section 10.2]{LOTBimodules}.
\end{proof}

\begin{figure}[ht]
\[
\xymatrix{
p_1 \ar[d]^{U^2} & p_2 \ar[d]^{U} \ar[l]_{\sigma_1} \ar[r]^{\sigma_3}
& p_3 \ar[d]^{U} \ar[r]^{\sigma_2} & p_4 \ar[d]^{U} \ar[r]^{\sigma_1} & p_5 \ar[d]^{U}
& x_0 \ar[l]^{U\sigma_{123}} \ar[dr]|{U^2 \sigma_1} & x_1 \ar[d]|{U^4} \ar[l]_{\sigma_2} \ar@/_1.5pc/[ll]_{U^3 \sigma_{23}}
& x_2 \ar[d]|{U^2} \ar[l]_{\sigma_3} \ar@/_4pc/[lllllll]_{U\sigma_{123}}
& x_3 \ar[d]|(0.45){U^2} \ar[dl]|{\sigma_2} \ar[ldddd]_(0.8){U\sigma_2} & x_4 \ar[d]^{U^2} \ar[l]_{\sigma_3} \ar@/^1pc/[ddddlllllll]_(0.6){U\sigma_{123}} & \\
q_1 & q_2 \ar[l]_{U \sigma_1} \ar[r]_{\sigma_3}
& q_3 \ar[r]_{\sigma_2} & q_4 \ar[r]_{\sigma_1} & q_5 &
& y_1 \ar[ll]_{\sigma_{23}} & y_2 \ar[l]_{U^2 \sigma_3} \ar@/^2.5pc/[lllllll]_{U\sigma_{123}}
& y_3 \ar@/^2pc/[ll]^{\sigma_{23}} \ar[ldddd]^{\sigma_2} & y_4 \ar[l]_{\sigma_3} \ar@/^1pc/[ddddlllllll]_(0.4){U\sigma_{123}}
& \\
\\ \\
& & r_1 \ar[d]^{U^2} & r_2 \ar[d]|(0.4){U} \ar[l]^{\sigma_1} \ar[r]^{\sigma_3}
& r_3 \ar[d]^{U} \ar[r]^(0.4){\sigma_2} & r_4 \ar[d]^{U} \ar[r]^{\sigma_1} & r_5 \ar[d]^{U}
& r_6 \ar[d]_{U} \ar[l]^{\sigma_{123}} \ar@/^1pc/[uuulllllll]^(0.6){\sigma_{123}} & r_7 \ar[d]_{U} \ar[l]_(0.4){\sigma_{12}}
& r_8 \ar[d]^{U} \ar[l]_{\sigma_2} & r_9 \ar[d]^{U} \ar[l]_{\sigma_3} \\
& & s_1 & s_2 \ar[l]^{U \sigma_1} \ar[r]_{\sigma_3} & s_3 \ar[r]_{\sigma_2} & s_4 \ar[r]_{\sigma_1} & s_5
& s_6 \ar[l]^{\sigma_{123}} & s_7 \ar[l]^{\sigma_{12}} & s_8 \ar[l]^{\sigma_2} & s_9 \ar[l]^{\sigma_3} \\
}
\]
\caption{A summand of $\CFDm(V,Q_{2,-1})$.}
\label{fig:CFDm(V,Q2-1)}
\end{figure}

\begin{proposition} \label{prop:CFDm(V,Q2-1)}
\allowdisplaybreaks
The type $D$ structure $\CFDm(V,Q_{2,-1})$ has a direct summand with a basis
\[
\{p_1, \dots, p_5, q_1, \dots, q_5, r_1, \dots, r_9, s_1, \dots, s_9, x_0, \dots, x_4, y_1, \dots, y_4\}
\]
with the following properties:
\begin{itemize}
\item
The generators
\[
p_2, p_4, q_2, q_4, r_2, r_4, r_7, r_9, s_2, s_4, s_7, s_9, x_0, x_2, y_2, x_4, y_4
\]
have are in $\iota_0 \CFDm(V,Q_{2,1})$, and the remaining ones are in $\iota_1 \CFDm(V,Q_{2,1})$.

\item
The differential is as shown in Figure \ref{fig:CFDm(V,Q2-1)}.
\end{itemize}
\end{proposition}

\begin{proof}
This proceeds similarly to the proof of Proposition \ref{prop:CFDm(V,Q21)}. Most of the changes of basis are the same; the primary difference is in the definitions of $r_7, r_8, r_9, s_7, s_8, s_9$, which use the $d$ and $e$ generators. The details are left to the reader as an exercise.
\end{proof}

\begin{proof}[Proof of Proposition \ref{prop:finalcases}]
For conciseness, we do not write down $\CFAm(V, Q_{2,-1})$ here. However, the reader can easily verify the following:

\begin{itemize}
\item
The nontrivial summand of $\CFAm(V, Q_{2,-1}) \boxtimes \CFD(X_O)$ is
\begin{equation} \label{eq:CFK(Q2-1(U))}
\xymatrix{
 & x_2 \otimes \xi_0 \ar[dl] \ar[d]^{U^2} & x_4 \otimes \xi_0 \ar[dl] \ar[d]^{U^2} \ar[r]^{U} & r_6 \otimes \xi_0 \ar[d]^{U} \\
x_0 \otimes \xi_0 & y_2 \otimes \xi_0 & y_4 \otimes \xi_0 \ar[r] & s_6 \otimes \xi_0.
}
\end{equation}
The homology is generated by $r_6 \otimes \xi_0$, which implies that $C_{x_0} = C_{x_2} = -3$, $C_{y_2} = C_{x_4} = -1$, $C_{r_6}=0$, and $C_{y_4} = C_{s_6} = 1$.

\item
When $\tau(K)<0$ and $\epsilon(K)=-1$, $\CFKm(Q(K)_{2,-1})$ has a summand $W$ whose differential is exactly the same as \eqref{eq:unstable21-tau<0,epsilon=-1}. The Alexander grading of the generator $y_2 \otimes \xi_0$ is now $-2\tau(K)-1$, so $\tau(Q(K)_{2,-1}) = 2\tau(K)+1 = 2 \tau(Q(K)) - 1$. Thus, $\epsilon(Q(K))=1$.

\item
When $\tau(K)=0$ and $\epsilon(K)=1$, the unstable chain in $\CFD(X_K)$ and a portion of the horizontal chain meeting $\xi_0 = \eta_2$ is
\[
\eta_1 \xrightarrow{D_3} \lambda \xrightarrow{D_2} \xi_0 \xrightarrow{D_{12}} \eta_0
\]
if the length of the horizontal arrow is $1$, and
\[
\lambda' \xrightarrow{D_{23}} \lambda \xrightarrow{D_2} \xi_0 \xrightarrow{D_{12}} \eta_0
\]
if the length is greater than $1$. In the former case, the relevant summand of $\CFKm(Q(K)_{2,-1})$ is
\begin{equation}
\xymatrix{
r_7 \otimes \eta_2 \ar[d]^{U} \ar[r] & r_6 \otimes \xi_0 \ar[d]^{U} & x_3 \otimes \lambda \ar[d]^{U^2} \ar[l]_{U} \ar[dr] & x_2 \otimes \xi_0 \ar[d]^{U^2} \ar[dr] & \\
s_7 \otimes \eta_2 \ar[r] & s_6 \otimes \xi_0 & y_3 \otimes \lambda \ar[l] & y_2 \otimes \xi_0 & x_0 \otimes \eta_0.
}
\end{equation}
In the latter case, the summand is the same, except that $r_7 \otimes \eta_2$ and $s_7 \otimes \eta_2$ are replaced with $r_8 \otimes \lambda'$ and $s_8 \otimes \lambda'$, respectively. In either case, the free part of the homology is generated by $y_3 \otimes \lambda$, with Alexander grading $1$. Hence, $\tau(Q(K)_{2,-1}) = -1$, so $\epsilon(Q(K))=1$. \qedhere
\end{itemize}
\end{proof}

\bibliography{bibliography}
\bibliographystyle{amsplain}

\end{document}